\documentclass[leqno,10pt]{amsart}

\usepackage[a4paper]{geometry}
\usepackage{color}
\usepackage[T1]{fontenc}
\usepackage{amssymb, bm}
\usepackage[english]{babel}
\usepackage{amssymb,amsthm}
\usepackage{mathtools}
\DeclareMathAlphabet{\mathcal}{OMS}{cmsy}{m}{n}
\usepackage[colorlinks=true,linkcolor=red,citecolor=blue]{hyperref}

\newcommand{\dd}{\partial}
\newcommand{\eps}{\varepsilon}
\newcommand{\2}{{L^2\!\left(\mathbb{T}^3\right)}}

\newcommand{\osc}{\textnormal{osc}}

\newcommand{\dx}{\textnormal{d}{x}}
\newcommand{\dive}{\textnormal{div}\;}
\newcommand{\diveh}{\textnormal{div}_h\;}

\newcommand{\nhp}{\nabla_h^\perp}

\newcommand{\nh}{\nabla_h}
\newcommand{\Dh}{\Delta_h}
\newcommand{\ck}{\check{k}}
\newcommand{\cm}{\check{m}}
\newcommand{\cn}{\check{n}}
\newcommand{\Lplus}{\mathcal{L}\left(\frac{t}{\varepsilon}\right)}
\newcommand{\Lminus}{\mathcal{L}\left(-\frac{t}{\varepsilon}\right)}
\newcommand{\Hs}{{{H}^s\left( \T^3 \right)}}
\newcommand{\R}{\mathbb{R}}

\newcommand{\uh}{\overline{u}^h}

\newcommand{\Tq}{\triangle_{q'}}
\newcommand{\sumf}{\sum_{\left| q-q' \right|\leqslant4}}
\newcommand{\sumi}{\sum_{q'>q-4}}
\def\L2h{\mathop{{L^2\left( \T^2_h \right)}}}
\newcommand{\tq}{\triangle_q}
\newcommand{\Sq}{S_q}

\newcommand{\set}[1]{\left\lbrace #1 \right\rbrace}

\newcommand{\pare}[1]{\left( #1 \right)}
\newcommand{\bra}[1]{\left[ #1 \right]}

\newcommand{\RN}[1]{%
  \textup{\uppercase\expandafter{\romannumeral#1}}%
}
\newcommand{\av}[1]{\left| #1 \right|}
\newcommand{\norm}[1]{\left\| #1 \right\|}

\newcommand{\PA}{\mathbb{P}\mathcal{A}}
\newcommand{\cA}{\mathcal{A}}
\newcommand{\cC}{\mathcal{C}}
\newcommand{\cL}{\mathcal{L}}
\newcommand{\cB}{\mathcal{B}}
\newcommand{\cQ}{\mathcal{Q}}
\newcommand{\cF}{\mathcal{F}}
\newcommand{\cJ}{\mathcal{J}}
\newcommand{\cD}{\mathcal{D}}
\newcommand{\cE}{\mathcal{E}}
\newcommand{\cS}{\mathcal{S}}
\newcommand{\cR}{\mathcal{R}}

\renewcommand{\v}{\textnormal{v}}

\newcommand{\bP}{\mathbb{P}}
\newcommand{\bC}{\mathbb{C}}

\newcommand{\ps}[2]{\psca{ \left. #1 \ \right| \ #2 }}
\newcommand{\psc}[2]{\left\langle \left. #1 \ \right| \ #2 \right\rangle}

\newcommand{\hra}{\hookrightarrow}

\newcommand{\loc}{\textnormal{loc}}
\newcommand{\T}{\mathbb{T}}

\newcommand{\bZ}{\mathbb{Z}}
\newcommand{\bS}{\mathbb{S}}

\newcommand{\NS}{Navier-Stokes }

\newcommand{\uU}{\underline{U}}
\newcommand{\uu}{\underline{u}^h}

\newcommand{\bU}{\overline{U}}
\newcommand{\tU}{\widetilde{U}}

\newcommand{\psca}[1]{\left\langle#1\right\rangle}

\newcommand{\cFLinfty}{L^\infty\left(\mathbb{T}^3\right)}
\newcommand{\cFLtwo}{{L^2\left(\mathbb{T}^3\right)}}
\newcommand{\cFLp}{{L^p\left(\mathbb{T}^3\right)}}

\newcommand{\cFHs}{{{H}^s\left( \T^3 \right)}}

\theoremstyle{theorem}
\newtheorem{theorem}{Theorem}[section]
\newtheorem*{theorem*}{Theorem}
\newtheorem{prop}[theorem]{Proposition}
\newtheorem{lemma}[theorem]{Lemma}
\newtheorem{cor}[theorem]{Corollary}

\newtheorem{definition}[theorem]{Definition}

\theoremstyle{definition}
\newtheorem{rem}[theorem]{Remark}

\numberwithin{equation}{section}

\begin{document}

\title[Density-dependent, incompressible periodic fluids]{On the influence of gravity on density-dependent\\ incompressible periodic fluids}

\author[V.-S. Ngo \& S. Scrobogna]{Van-Sang Ngo \& Stefano Scrobogna}

\date{\today}

\address{\noindent \textsc{V.-S. Ngo,  Universit\'e de Rouen, Laboratoire de Math\'ematiques Rapha\"el Salem, UMR 6085 CNRS, 76801 Saint-Etienne du Rouvray, France}}
\email{van-sang.ngo@univ-rouen.fr}

\address{\noindent \textsc{S. Scrobogna, Basque Center for Applied Mathematics, Mazarredo, 14, E48009 Bilbao, Basque Country, Spain}}
\email{sscrobogna@bcamath.org}


\begin{abstract}
	The present work is devoted to the analysis of density-dependent, incompressible fluids in a 3D torus, when the Froude number $\eps$ goes to zero. We consider the very general case where the initial data do not have a zero horizontal average, where we only have smoothing effect on the velocity but not on the density and where we can have resonant phenomena on the domain. We  explicitly determine the limit system when $\eps \to 0$ and prove its global wellposedness. Finally, we prove that for large initial data, the density-dependent, incompressible fluid system is globally wellposed, provided that $\eps$ is small enough.
\end{abstract}

\maketitle


\section{Introduction}

In this paper, we want to study the motion of an incompressible, inhomogeneous fluid whose density profile is considered to be a perturbation around a stable state, which is describe be the following system 
\begin{equation} \label{PBSe} \tag{PBS$_\varepsilon$}
	\left\lbrace
	\begin{aligned}
		&\partial_t v^\varepsilon  + v^\varepsilon \cdot\nabla v^\varepsilon  -\nu \Delta  v^\varepsilon -\displaystyle\frac{1}{\varepsilon} \rho^\varepsilon \overrightarrow{e}_3 &=& -\displaystyle \frac{1}{\varepsilon} \nabla \Phi^\varepsilon,\\
		&\partial_t \rho^\varepsilon + v^\varepsilon \cdot\nabla \rho^\varepsilon   + \displaystyle\frac{1}{\varepsilon} v^{3,\varepsilon} & =& \;0,\\
		&\dive v^\varepsilon =\;0,\\
		& \left. \left( v^\varepsilon, \rho^\varepsilon \right)\right|_{t=0}= \left( v_0, \rho_0 \right), 
	\end{aligned}
	\right.
\end{equation}
in the regime where the Froude number $ \varepsilon\to 0 $. Here, the vector field $v^\eps$ and the scalar function $\rho^\eps$ represent respectively the velocity and the density of the fluid and $\nu$ stands for the viscosity. For a more detailed discussion about the physical motivation and the derivation of the model, we refer to \cite{Scrobo_Froude_periodic},  \cite{Scrobo_Froude_FS}, \cite{Scrobo_thesis}. We also refer to the monographs \cite{cushman2011introduction} and \cite{Pedlosky87} for a much wider survey on geophysical models.

Let us give some brief comments about the system \eqref{PBSe}. In a nutshell, there are two forces which constrain the motion of a fluid on a geophysical scale: the Coriolis force and the gravitational stratification. The predominant influence of one force, the other, or both gives rise to substantially different dynamics.

The \emph{Coriolis force} (see \cite{SM89} for a detailed analysis of such force) is due to the rotation of the Earth around its axis and acts perpendicularly to the motion of the fluid. If the magnitude of the force is sufficiently large (when the rotation is fast or the scale is large for example), the Coriolis force ``penalizes'' the vertical dynamics of the fluid and makes it move in rigid columns (the so-called \emph{Taylor columns}). This tendency of a rotating fluid to displace in vertical homogeneous columns is generally known as \textit{Taylor-Proudmann theorem}, which was first derived by Sidney Samuel Hough (1870-1923), a mathematician at Cambridge in the work \cite{Hough1897}, but it was named after the works of G.I. Taylor \cite{Taylor1917} and Joseph
Proudman \cite{Proudman1916}. On a mathematical point of view, the Taylor-Proudman effect for homogeneous, fast rotating fluids is a rather well understood after the works \cite{BMN1}, \cite{BMNresonantdomains}, \cite{CDGG2}, \cite{Gallagher_schochet} and  \cite{grenierrotatingeriodic}. In such setting, we mention as well the work \cite{GallagherSaint-Raymondinhomogeneousrotating} in which, the authors consider an inhomogeneous rotation and the works \cite{FGG-VN}, \cite{FGN} and \cite{Scrobo_Ngo} in which fast rotation was considered simultaneously with weak compressibility.

Beside the rotation, one can consider a fluid which is inhomogeneous and whose density profile is a linearization around a stable state, we refer to \cite{Charve_thesis}, \cite{cushman2011introduction} and references therein for a thorough derivation of the model. In such situation, we can imagine that the rotation effects and stratification effects are equally relevant: the system describing such effect is known as \textit{primitive equations} (see \cite{Charve_thesis} and \cite{cushman2011introduction}). The primitive equations and their asymptotic dynamic as stratification and rotation tend to infinity at a comparable rate are as well rather well understood on a mathematical viewpoint: we refer to the works \cite{charve1}, \cite{charve2},  \cite{charve4}, \cite{charve3}, \cite{charve_ngo_primitive}, \cite{chemin_prob_antisym}, \cite{Gallagher_schochet}, \cite{Scrobo_primitive_horizontal_viscosity_periodic} and references therein.

\medskip

Now, we will briefly discuss the \emph{gravitational stratification} effect, which is the main physical phenomenon concerning the system \eqref{PBSe} for a inhomogeneous fluid, subjected to a gravitational force pointing downwards. Gravity force tends to lower the regions of the fluid with higher density and raise the regions with lower density, trying finally to dispose the fluid in horizontal stacks of vertically decreasing density. A fluid in such configuration (density profile which is a decreasing function of the variable $ x_3 $ only) is said to be in a \textit{configuration of equilibrium}.

Let hence consider a fluid in a configuration of equilibrium and let us imagine to raise a small parcel of the fluid with high density in a region of low density. Since such parcel is much heavier (in average) than the fluid surrounding it, the gravity force will induce a downwards motion. Such motion does not stop until the parcel reaches a layer whose density is comparable to its own, and inertially it will continue to move downwards until sufficient buoyancy is provided to invert the motion, due to Archimedes principle. This kind of perturbation of an equilibrium state induces hence a pulsating motion which is described by the linear application
\begin{equation} \label{eq:stratification_buoyancy}
	\pare{u^{1, \varepsilon}, u^{2, \varepsilon}, u^{3, \varepsilon}, \rho^{\varepsilon}} \mapsto \frac{1}{\varepsilon} \pare{0, 0, -\rho^{\varepsilon}, u^{3, \varepsilon}}, 
\end{equation}
which appears in \eqref{PBSe}. The application \eqref{eq:stratification_buoyancy} is called \textit{stratification buoyancy} and we will base our analysis on the dispersive effects induced by such perturbation.

To the best of our knowledge, there are not many results concerning the effects of the stratification buoyancy. In \cite{embid_majda2}, there was a first attempt to perform a multiscale analysis when Rossby and Froude number are in different regimes, while in \cite{Scrobo_Froude_FS} and \cite{Widmayer_Boussinesq_perturbation}, the authors studied the convergence and stability of solutions of \eqref{PBSe} when the Froude number $ \varepsilon\to 0 $ in the whole space $ \mathbb{R}^3 $. In \cite{Scrobo_Froude_periodic} the system \eqref{PBSe} is studied in nonresonant domains when the initial data has zero horizontal average.

\medskip

In this paper, the unknowns $\pare{ v^\varepsilon, \rho^\varepsilon }$ are considered to be functions in the variables $\pare{ x, t }\in \T^3\times \R_+ $ being the space domain $\T^3$ the three-dimensional periodic box
\begin{equation*}
\T^3 = \prod_{i=1}^3 \R \left/ a_i \mathbb{Z} \right. , \hspace{1cm} a_i \in \R.
\end{equation*}
Compared to \cite{Scrobo_Froude_periodic}, we consider the much more general case with the following additional difficulties
\begin{enumerate}
	\item \label{enumerate:punto1} Initial data are considered with generic horizontal average. This point seem marginal, but as showed in this paper, the dynamics induced by initial data with nonzero horizontal average create additional vertical gravitational perturbations, the control of which is highly non-trivial (see as well \cite{GS3}).
	\item \label{enumerate:punto2} Generic space domain may present resonant effects.
	\item \label{enumerate:punto3} Density profiles are only transported and do not satisfy a transport-diffusion equation and so do not possess smoothing effects.
\end{enumerate}

\noindent From now on we rewrite the system \eqref{PBSe} in the following more compact form
\begin{equation} \tag{\ref{PBSe}}
	\left\lbrace
	\begin{aligned}
		& {\partial_t V^\varepsilon}+ v^\varepsilon \cdot \nabla V^\varepsilon - \cA_2\pare{D}V^\varepsilon + \frac{1}{\varepsilon}\cA V^\varepsilon = -\frac{1}{\varepsilon} \left( \begin{array}{c} \nabla \Phi^\varepsilon\\ 0 \end{array} \right),\\
		& V^\varepsilon=\left( v^\varepsilon,\theta^\varepsilon \right),\\
		&\dive v^\varepsilon=0,\\
		 & \left. V^\varepsilon \right|_{t=0}= V_0,
	\end{aligned}
	\right.
\end{equation}
 where
\begin{align}\label{matrici}
\mathcal{A}= & \left( \begin{array}{cccc}
0&0&0&0\\
0&0&0&0\\
0&0&0&1\\
0&0&-1&0\\
\end{array} \right),
&
\cA_2\pare{D} =& \left( \begin{array}{cccc}
\nu\Delta&0&0&0\\
0&\nu\Delta&0&0\\
0&0&\nu\Delta&0\\
0&0&0&0\\
\end{array} \right).
\end{align}
The additional difficulties \eqref{enumerate:punto1}--\eqref{enumerate:punto3} listed above are the main difficulties in the present work and they modify significantly the dynamic of \eqref{PBSe} compared to the results proved in \cite{Scrobo_Froude_periodic}, as already mentioned. Let us hence start describing the effects induced by the hypothesis made in the point \ref{enumerate:punto1}: we will see in the following that the dynamics of the solutions of \eqref{PBSe} in the limit regime $ \varepsilon \to 0 $ is essentially governed by the effects of the outer force $ \varepsilon^{-1}\cA V^{\varepsilon} $.

\subsection{A survey on the notation adopted.}\label{sec:notation_and_results}

All along this note we consider real valued vector fields, i.e. applications $ V:\R_+\times \T^3 \to  \R^4 $. We will often associate to a vector field $V$ the vector field $v$ which  shall be simply the projection on the first three components of $V$. The vector fields considered are periodic in all their directions and they have zero global average $ \int_{\T^3} V \dx=0 $, which is equivalent to  assume that the first Fourier coefficient $ \hat{V} \left( 0 \right)=0 $. We remark that the zero average propriety stated above is preserved in time $t$ for both \NS\ equations as well as for the system \eqref{PBSe}.\\
Let us define the Sobolev space $ {\cFHs} $, which consists in all the tempered distributions $ u $ such that
\begin{equation}
\label{eq:non-hmogeneous_Sobolev_norms}
\left\| u \right\|_{\cFHs}= \left( \sum_{n\in\mathbb{Z}^3}\left(1+ \left| \cn \right|^{2} \right)^{s}\left| \hat{u}_n \right|^2 \right)^{1/2}<\infty.
\end{equation}
Since we shall consider always vector fields whose average is null the Sobolev norm defined above in particular is equivalent to the following semi-norm
$$
\left\| \left( -\Delta \right)^{s/2} u \right\|_{\cFLtwo} \sim \left\| u \right\|_{\cFHs}, \hspace{1cm}s\in\R,
$$
which appears naturally in parabolic problems.\\
Let us define the operator $\mathbb{P}$ as the three dimensional Leray operator $\mathbb{P}^{(3)}$ wich leaves untouched the fourth component, i.e.
\begin{equation}
	\label{eq:newLeray} \mathbb{P}= \left(
	\begin{array}{c|c}
	1-{\Delta^{-1}}{\partial_i\partial_j}& 0\\  \hline
	0 & 1
	\end{array}\right)_{i,j=1,2,3}= \left( \begin{array}{c|c}
	\mathbb{P}^{(3)} & 0 \\ \hline
	0 & 1
	\end{array} \right).
\end{equation}
The operator $ \mathbb{P} $ is  a pseudo-differential operator, in the Fourier space its symbol is
\begin{equation}
\label{Leray projector}
\mathbb{P}_n= \left(
\begin{array}{c|c}
\delta_{i,j}-\dfrac{\check{n}_i\; \check{n}_j}{ \left| \check{n} \right|^2}& 0\\[4mm]  \hline 
0 & 1
\end{array}\right)_{i,j=1,2,3},
\end{equation}
where $ \delta_{i,j} $ is Kronecker's delta and $ \check{n}_i=n_i/a_i, \left| \check{n} \right|^2= \sum_i \check{n}_i^2$.\\

\subsection{Results}

Being the operator $\cA$ skew-symmetric it is possible to apply energy methods to the system \eqref{PBSe} in the same fashion as it is done in \cite[Chapter 4]{BCD} for quasilinear symmetric hyperbolic systems. Being this the case we can deduce the following local existence result

\begin{theorem}\label{thm:local}
Let $V_0\in \Hs$ where $s>3/2$, there exist a $T^\star >0$ such that for every $T\in \left[0, T^\star \right)$ the system \eqref{PBSe} admits a unique solution in the energy space
\begin{equation*}
\cC \pare{  \left[0, T \right]; \Hs }\cap \cC^1 \pare{  \left[0, T \right]; H^{s-1}\pare{\T^3} }.
\end{equation*}
Moreover there exist a positive constant $ c $ such that
\begin{equation*}
T > \frac{c}{\norm{V_0}_{\Hs}},
\end{equation*}
and the maximal time of existence $ T^\star $ is independent of $ \varepsilon $ and $ s $ and, if $ T^\star < \infty $, then
\begin{equation}\label{eq:BU_criterion}
\int_0^{T^\star} \norm{\nabla U^\varepsilon \pare{t}}_{L^\infty \pare{\T^3}}\textnormal{d} t < \infty.
\end{equation}
\end{theorem}

From now on we adopt the following notation: let us consider a vector field $A$, we denote as $\underline{A}= \underline{A} \pare{ x_3 } $ the horizontal average of $A$ defined as
\begin{equation*}
\underline{A} \pare{ x_3 } = \frac{1}{4\pi^2 a_1a_2} \int_{\T^2_h} A \pare{ y_h, x_3  } \textnormal{d} y_h,
\end{equation*} and by $\tilde{A}=\tilde{A}\pare{ x_h, x_3 }$ the horizontal oscillation of $A$ defined as
\begin{equation}\label{eq:bar-tilde_dec}
\tilde{A}\pare{ x_h, x_3 }= A \pare{ x_h, x_3 } - \underline{A}\pare{ x_3 }.
\end{equation}

Let us now define the operator 
\begin{equation}\label{eq:def_L}
\cL \pare{ \tau }= e^{-\tau \PA},
\end{equation}
where the matrices $\bP$ and $\cA$ are defined respectively in \eqref{Leray projector} and \eqref{matrici}, the result we prove is the following;

\begin{theorem}\label{thm:main_result}
Let $V_0$ be in $\Hs$ for $s>9/2 \pare{=\frac{d}{2}+3}$. Let us define 
\begin{align*}
\underline{V_0} & = \frac{1}{2\pi^2 a_1 a_2} \int_{\T^2_h} V_0\pare{y_h, x_3}\textnormal{d} y_h, \\
\uh_0 & = \pare{\begin{array}{c}
-\partial_2\\ \partial_1
\end{array} } \pare{-\Dh}^{-1} \pare{-\partial_2 V_0^1+\partial_1 V_0^2}, \\
\bU_0 & = \pare{\uh_0, 0, 0}^\intercal, \\
U_{\osc, 0} & = V_0 - \bU_0 - \underline{V_0}.
\end{align*}
 There exists a $ \varepsilon_0 > 0 $ such that for each $ \varepsilon\in \pare{0, \varepsilon_0} $ 
\begin{equation*}
V^\varepsilon = \underline{U} + \bar{U} + \Lplus U_{\osc} + o_\varepsilon \pare{ 1 }\hspace{1cm} \text{in } \ \cC_{\loc} \pare{ \R_+; H^{s-2}\pare{\T^3}},
\end{equation*}
where $\underline{U}= \pare{ \uu, 0, \uU^4 }= \pare{ \uU^1, \uU^2, 0, \uU^4 } $, $ \bar{U}= \pare{ \uh, 0, 0 } = \pare{ \bar{u}^1, \bar{u}^2, 0, 0 }$ and $U_{\osc}$ solve the systems
\begin{align}\label{eq:Uunder_heat2D}
&\left\lbrace
\begin{aligned}
& \partial_t \uu \left( x_3, t \right) - \nu \partial_3 ^2 \uu \left( x_3 , t \right) =0,\\
&  \underline{U}^4 \pare{ x_3, t } = \underline{V_0^4},\\
& \left. \uu \right|_{t=0}= \underline{V_0^h},\\
& \left. \underline{U}^4 \right|_{t=0}= \underline{V_0^4},
\end{aligned}
\right.
\\
&\label{eq:Uover_2DstratNS}
\left\lbrace
\begin{aligned}
&
\begin{multlined}
 \partial_t \uh \left( t,  x_h, x_3 \right) + \uh\left( t,  x_h, x_3 \right) \cdot \nh \uh \left( t,  x_h, x_3 \right) + \uu \left( t, x_3 \right)\cdot \nh \uh \left( t, x_h, x_3 \right)
\\ 
  - \nu \Delta \uh\left( t,  x_h, x_3 \right) = -\nh \bar{p} \left( t,  x_h, x_3 \right)
\end{multlined}  
  \\
   & \diveh \uh \left( x_h, x_3 \right) =0,\\
& \left. \uh\left( t,  x_h, x_3 \right) \right|_{t=0}=\uh_0\left(   x_h, x_3 \right).
\end{aligned}
\right.
\\
&\label{eq:lim_Uosc}
\left\lbrace
\begin{aligned}
&\partial_t U_{\osc} +\widetilde{\mathcal{Q}}_1\left(U_{\osc} +2 \bar{U}, U_{\osc} \right) +\mathcal{B} \left( \uU, U_{\osc} \right) -  \nu \Delta U_{\osc} =0,\\
& \dive U_{\osc}=0, \\
&\left. U_{\osc} \right|_{t=0}=U_{{\osc}, 0}.
\end{aligned}
\right.
\end{align}
Where $\cB$ and $\cQ$ are specific localizations of the transport form, whose explicit expression is given by the equations \eqref{eq:tQ2_e+-} and \eqref{eq:limit_cQ1} and are here omitted for the sake of clarity. 
\end{theorem}

\begin{rem} \label{rem:propagation_parabolicity}
It is interesting to remark that, despite  $V^\varepsilon$ is the solution of the parabolic-hyperbolic system \eqref{PBSe},  it is  well defined in the space $\cC \pare{ [0,T]; H^s }$ for $s>9/2$ and $ T>0 $. This improvement of regularity is known as propagation of parabolicity: such name is motivated by the fact that the limit equations \eqref{eq:Uover_2DstratNS} and \eqref{eq:lim_Uosc} are strictly parabolic and we can prove they are globally well posed in some suitable energy space of subcritical regularity. A similar phenomenon takes place as well in the study of the incompressible limit for weakly compressible fluids, we refer the reader to the works \cite{DanchinCompressible} and \cite{Gallagher_incompressible_limit}.
\end{rem}

\begin{rem}\label{rem:smoothness_bilinear_Fourier}
Equation \eqref{eq:lim_Uosc} is a nonlinear three-dimensional parabolic equation. The nonlinearity $\widetilde{\cQ}_1$ is a modified, symmetric transport form which assumes the following explicit form
\begin{equation*}
\widetilde{\cQ}_1 \pare{ A, B } = \frac{1}{2} \chi \pare{ D } \pare{ A\cdot \nabla B + B \cdot \nabla A },
\end{equation*}
where $\chi$ is a Fourier multiplier of order zero which localizes bilinear interactions on a very specific frequency set. Following the theory of the three-dimensional \NS\ equations we do not expect hence \eqref{eq:lim_Uosc} to be globally well-posed. Despite this, we shall see that the zero-order Fourier multiplier $ \chi\left( D \right) $ has in fact a nontrivial smoothing effect, which makes the bilinear interaction $ \widetilde{\cQ}_1 \left( U_{\osc}, U_{\osc} \right) $ smoother than $ U_{\osc}\cdot \nabla U_{\osc} $.
\end{rem}

\subsection{Elements of Littlewood-Paley theory.}\label{elements LP}

A tool that will be widely used all along the paper is the theory of Littlewood--Paley, which consists in doing a dyadic cut-off of the  frequencies.\\
Let us define the (non-homogeneous)  truncation operators as follows:
\begin{align*}
\tq u= & \sum_{n\in\mathbb{Z}^3} \hat{u}_n \varphi \left(\frac{\left|\check{n}\right|}{2^q}\right) e^{i\check{n} \cdot x}, &\text{for }& q\geqslant 0,\\
\triangle_{-1}u=& \sum_{n\in\mathbb{Z}^3} \hat{u}_n \chi \left( \left|\check{n} \right| \right)e^{i\check{n} \cdot x},\\
\tq u =& 0, &\text{for }& q\leqslant -2,
\end{align*}
where $u\in\mathcal{D}'\left(\mathbb{T}^3 \right)$ and  $\hat{u}_n$ are the Fourier coefficients of $u$. The functions $\varphi$ and $\chi$ represent a partition of the unity in $\mathbb{R}$, which means that are smooth functions with compact support such that
\begin{align*}
\text{supp}\;\chi \subset&\; B \left(0,\frac{4}{3}\right), & \text{supp}\;\varphi \subset& \;\mathcal{C}\left( \frac{3}{4},\frac{8}{3}\right),
\end{align*}
and such that for all $t\in\mathbb{R}$,
$$
\chi\left( t\right) +\sum_{q\geqslant 0} \varphi \left( 2^{-q}t\right)=1.
$$

Let us define further the low frequencies cut-off operator
$$
S_q u= \sum_{q'\leqslant q-1}\Tq u.
$$

\subsubsection{Paradifferential calculus.}\label{paradifferential calculus}

The dyadic decomposition turns out to be very useful also when it comes to study the product betwee two distributions. We can in fact, at least formally, write for two distributions $u$ and $v$
\begin{align}\label{decomposition vertical frequencies}
u=&\sum_{q\in\mathbb{Z}} \tq u ; &
v=&\sum_{q'\in\mathbb{Z}} \Tq v;&
u\cdot v = & \sum_{\substack{q\in\mathbb{Z} \\ q'\in\mathbb{Z}}}\tq u \cdot \Tq v.
\end{align}

We are going to perform a Bony decomposition (see  \cite{BCD}, \cite{Bony1981}, \cite{chemin_book} for the isotropic case and \cite{CDGG2},\cite{iftimie_NS_perturbation} for the anisotropic one).
\\
Paradifferential calculus is  a mathematical tool for splitting the above sum in three parts 
\begin{itemize}
\item The first part concerns  the indices $\left(q,q'\right)$ for which the size of  $\text{ supp}\;\mathcal{F}\left( \tq u\right)$ is small compared to $\text{supp}\;\mathcal{F}\left( \Tq v\right)$.
\item The second part contains the indices corresponding to those frequencies of $u$ which are large compared to those of $v$.
\item In the last part  $\text{supp}\;\mathcal{F}\left( \Tq v\right)$ and $\text{ supp}\;\mathcal{F}\left( \tq u\right)$ have comparable sizes.
\end{itemize}
In particular we obtain
$$
u\cdot v = T_u v+ T_v u + R\left(u,v\right),
$$
where
\begin{align*}
T_u v=& \sum_q S_{q-1} u\; \tq v,&
 T_v u= & \sum_{q'} S_{q'-1} v \; \Tq u,&
 R\left( u,v \right) = & \sum_k \sum_{\left| \nu\right| \leqslant 1} \triangle_k  u\; \triangle_{k+\nu} v.
\end{align*}
The following almost orthogonality properties hold
\begin{align*}
\tq \left( \Sq a \Tq b\right)=&0, & \text{if }& \left|q-q'\right|\geqslant 5, \\
\tq \left( \Tq a \triangle_{q'+\nu}b\right)=&0, & \text{if }& q'> q-4,\; \left| \nu \right|\leqslant 1,
\end{align*}
and hence we will often use the following relation
\begin{align}
\tq\left( u\cdot v \right)= &\sum_{\left| q -q'\right| \leqslant 4} \tq\left(S_{q'-1} v\; \Tq u\right) +
\sum_{\left| q -q'\right| \leqslant 4} \tq\left(S_{q'-1} u\; \Tq v\right)+
\sum_{q'\geqslant q-4}\sum_{|\nu|\leqslant 1}\tq\left(  \Tq a \triangle_{q'+\nu}b\right)\nonumber ,\\
=& \sum_{\left| q -q'\right| \leqslant 4} \tq\left(S_{q'-1} v \; \Tq u\right) + \sum_{q'>q-4} \tq\left( S_{q'+2} u \Tq v\right).\label{Paicu Bony deco}
\end{align}

There is an interesting relatoin of regularity between dyadic blocks and full function in the Sobolev spaces, i.e.
\begin{equation}
\label{regularity_dyadic}
\left\| \tq f \right\|_{\cFLtwo} \leqslant C c_q (f) 2^{-qs}\left\| f \right\|_{\cFHs},
\end{equation}
with $ \left\| \left\lbrace c_q \right\rbrace_{q\in\mathbb{Z}} \right\|_{\ell^2\left( \mathbb{Z} \right)}\equiv 1 $. In the same way we denote as $ b_q $ a sequence in $ \ell^1 \left( \mathbb{Z} \right) $ such that $ \sum_q \left| b_q \right| \leqslant 1$.\\

The interest in the use of the dyadic decomposition is that the derivative of a function localized in  frequencies of size $2^q$ acts like the multiplication with the factor $2^q$ (up to a constant independent of $q$). In our setting (periodic case) a Bernstein type inequality holds. For a proof of the following lemma in the anisotropic (hence as well isotropic) setting we refer to the work \cite{iftimie_NS_perturbation}. For the sake of self-completeness we state the result in both isotropic and anisotropic setting.

\begin{lemma}\label{bernstein inequality}
Let $u$ be a function such that $\mathcal{F}u $ is supported in $ 2^q\mathcal{C}$, where $\mathcal{F}$ denotes the Fourier transform. For all integers $k$ the following relation holds
\begin{align*}
2^{qk}C^{-k}\left\| u \right\|_{{\cFLp}}\leqslant & \left\|\left( -\Delta \right)^{k/2} u \right\|_{{\cFLp}} \leqslant 2^{qk}C^{k}\left\| u \right\|_{{\cFLp}}.
\end{align*}

Let now $r\geqslant r' \geqslant 1$ be real numbers. Let  $\text{supp}\mathcal{F}u \subset  2^q B$, then
\begin{align*}
\left\| u \right\|_{ L^r}\leqslant & C \cdot 2^{3q\left( \frac{1}{r'}-\frac{1}{r}\right)}\left\| u \right\|_{ L^{r'}}.
\end{align*}
Let us consider now a function $ u $ such that $ \mathcal{F}u $ is supported in $ 2^q\mathcal{C}_h \times 2^{q'}\mathcal{C}_v $. Let us define $ D_h= \left( -\Dh \right)^{1/2}, D_3=\left| \partial_3 \right| $, then
$$
C^{-q-q'}2^{qs+q's'}\left\| u \right\|_{\cFLp}\leqslant 
\left\| D_h^s D_3^{s'} u \right\|_{\cFLp} \leqslant C^{q+q'}2^{qs+q's'}\left\| u \right\|_{\cFLp},
$$
and given $ 1\leqslant p'\leqslant p\leqslant \infty $, $1\leqslant r'\leqslant r\leqslant \infty $, then 
\begin{align*}
\left\| u \right\|_{L^p_hL^r_v} \leqslant & C^{q+q'} 2^{2q \left( \frac{1}{p'}-\frac{1}{p} \right) + q' \left( \frac{1}{r'}-\frac{1}{r} \right)
} \left\| u \right\|_{L^{p'}_h L^{r'}_v},\\
\left\| u \right\|_{L^r_vL^p_h} \leqslant & C^{q+q'} 2^{2q \left( \frac{1}{p'}-\frac{1}{p} \right) + q' \left( \frac{1}{r'}-\frac{1}{r} \right)
} \left\| u \right\|_{ L^{r'}_vL^{p'}_h}.
\end{align*}
\end{lemma}

\section{Analysis of the linear perturbation operator $\PA$}

\label{sec:linear_problem}

We recall that thoughout this paper, we alway use upper-case letters to represent vector fields on $\T^3$, with four components, the first three components of which form a divergence-free vector field (denoted by the same lower-case letter). More precisely, for a generic vector field $A$, we have 
\begin{equation*}
	A(x_1,x_2,x_3) = \pare{A^1(x_1,x_2,x_3),A^2(x_1,x_2,x_3),A^3(x_1,x_2,x_3),A^4(x_1,x_2,x_3)} = \pare{a(x_1,x_2,x_3), A^4(x_1,x_2,x_3)},
\end{equation*}
where
\begin{equation*}
	a = (a^1,a^2,a^3) \equiv (A^1,A^2,A^3), \quad\mbox{and}\quad \dive a = 0.
\end{equation*}

As explained in the introduction, the time derivative $\dd_t V^\eps$ is not uniformly bounded. In order to take the limit $\eps \to 0$, we need to filter the high oscillating terms out of the system \eqref{PBSe}. To this end, we consider the following linear, homogeneous Cauchy problem, which describes the internal waves associated to \eqref{PBSe}
\begin{equation}
	\label{eq:lin_prob}
	\left\lbrace
	\begin{aligned}
		&\partial_\tau W + \PA \; W = 0,\\
		&W|_{\tau=0}=W_0\in L^2_\sigma(\T^3),
	\end{aligned}
	\right.
\end{equation}
where 
\begin{equation*}
	L^2_\sigma \stackrel{def}{=} \set{ U=\pare{u, U^4}\in L^2, \; \dive u = 0 \;\mbox{ and } \int_{\T^3} U(x)\ dx = 0 }.
\end{equation*}
In \cite{Scrobo_Froude_periodic}, a detailed analysis of \eqref{eq:lin_prob} was given in a particular case where the vector fields are supposed to have zero horizontal average. In this section, we will provide a complete spectral analysis of the operator $\PA$ in the general case where there is no such assumption, which allows to get a detailed description of the solution of \eqref{eq:lin_prob} in $L^2_\sigma(\T^3)$. Using the decomposition \eqref{eq:bar-tilde_dec}, we can write
\begin{equation*}
	{L^2_\sigma} = \widetilde{L^2_\sigma} \oplus \underline{{L^2_\sigma}},
\end{equation*}
where
\begin{align*}
	\widetilde{L^2_\sigma}= & \Big\{ U=\pare{u, U^4}\in L^2_{\sigma} \ \Big\vert \ \int_{\T^2_h} U\pare{x_h, x_3}\, dx_h=0  \Big\},   \\
 	\underline{L^2_\sigma} = & \Big\{U=\pare{u, U^4 } \in L^2_\sigma \ \Big\vert \ U=U\pare{ x_3 } \ \text{and} \ u^3\equiv 0 \Big\}.
\end{align*}
Let us remark that, since $\dive a = 0$ and $a = a(x_3)$, we deduce that $\partial_3 a^3=0$, which in turn implies that $a^3=0$, taking into account the zero average of $a^3$ in $\T^3$.

\medskip 

Writing the first equation of \eqref{eq:lin_prob} in the Fourier variables, we have
\medskip
\begin{equation}
	\label{eq:lin_probF} 
	\left\lbrace
	\begin{aligned}
		&\dd_\tau \widehat{W} (\tau,n) + \widehat{\PA}\,(n)\; \widehat{W}(\tau,n) = 0,\\
		&\widehat{W}(0,n) = \widehat{W}_0(n),
	\end{aligned}
	\right.
\end{equation}
where 
\begin{equation}
	\label{eq:def_PA}
	\widehat{\PA}\,(n) = \left( \begin{array}{cccc}
				0&0&0&-\frac{ \cn_1 \cn_3}{\left| \cn \right|^2}\\
				0&0&0&-\frac{ \cn_2 \cn_3}{\left| \cn \right|^2}\\
				0&0&0&1-\frac{ \cn_3^2}{\left| \cn \right|^2}\\
				0&0&-1&0
			\end{array} \right).
\end{equation}
Standard calculations show that the matrix $\widehat{\PA}\,(n)$ possesses very different spectral properties in the case where $\cn_h = 0$ and in the case where $\cn_h \neq 0$.

\emph{1. In the case where $\cn_h \neq 0$:} The matrix $\widehat{\PA}\,(n)$ admits an eigenvalue $\omega^0(n)\equiv 0$ of multiplicity 2 and two other conjugate complex eigenvalues
\begin{equation}
	\label{eq:eigenvalues} i \omega^\pm (n) = \pm i \omega(n),
\end{equation}
where $\omega(n) = \frac{\left|\cn_h\right|}{\left|\cn\right|}$. Associated to each eigenvalue, there is a unique unit eigenvector, orthogonal to the frequency vector ${}^t(\cn_1,\cn_2,\cn_3,0)$, which is explicitly given as follows
\begin{align}
	\label{eq:eigenvectors}
	e_0(n) = &\frac{1}{\left|\cn_h\right|} \left( \begin{array}{c} -\cn_2\\ \cn_1\\ 0\\ 0 \end{array}\right), 
	&e_\pm(n)= &\frac{1}{\sqrt{2}} \left( \begin{array}{c}
		\pm \; i \; \frac{\cn_1 \cn_3}{\left| \cn_h \right|\; \left| \cn \right|}\\[2mm]
		\pm \; i \; \frac{\cn_2 \cn_3}{\left| \cn_h \right|\; \left| \cn \right|}\\[2mm]
		\mp \; i \; \frac{\left| \cn_h \right|}{\left| \cn \right|}\\[2mm]
		1
	\end{array}\right).
\end{align}	
Since $\set{e_\alpha}_{\alpha = 0,\pm}$ form an orthonormal basis of the subspace of $\mathbb{C}^4$ which is orthogonal to ${}^t(\cn_1,\cn_2,\cn_3,0)$, the classical theory of ordinary differential equations imply that the solution $\widehat{W}(\tau,n)$ of \eqref{eq:lin_probF}, with $\cn_h \neq 0$, writes
\begin{equation}
	\label{eq:FWnh} \widehat{W} (\tau,n) = \sum_{\alpha \in \set{0,\pm}} e^{i\tau \omega^\alpha(n)} \psca{\widehat{W}_0(n), e_\alpha(n)}_{\mathbb{C}^4} e_\alpha(n).
\end{equation}	
	
\medskip

\emph{2. In the case where $\cn_h = 0$:} The matrix $\widehat{\PA} (n) $ becomes
\begin{equation*}
	\widehat{\PA}\left( 0,0,n_3 \right)= \left( 
	\begin{array}{cccc}
	0&0&0&0\\
	0&0&0&0\\
	0&0&0&0\\
	0&0&-1&0
	\end{array}
	 \right),
\end{equation*}
and admits only one eigenvalue $\omega (0,0,n_3) = 0$ of multiplicity 4, and three associate unit eigenvectors, orthogonal to the frequency vector ${}^t(0,0,\cn_3,0)$
\begin{align}
	\label{ev_nh=0}
	f_1 = &	\left( \begin{array}{c} 1\\0\\0\\0 \end{array} \right),
	&
	f_2 = & \left( \begin{array}{c} 0\\1\\0\\0 \end{array} \right),
	&
	f_3 = & \left( \begin{array}{c} 0\\0\\0\\1 \end{array} \right).
\end{align}
The solution $\widehat{W}(\tau,n_3) \equiv \widehat{W}(\tau,0,0,n_3)$ of \eqref{eq:lin_probF} writes
\begin{equation}
	\label{eq:FWnh0} \widehat{W} (\tau,n_3) = \sum_{j \in \set{1,2,3}} \psca{\widehat{W}_0(0,0,n_3), f_j}_{\mathbb{C}^4} f_j.
\end{equation}

\medskip

The expressions of the Fourier modes $\widehat{W}(n)$ given in \eqref{eq:FWnh} and \eqref{eq:FWnh0} imply the following result

\begin{lemma}
	\label{lem:DecompW} Let $W_0\in L^2_\sigma(\T^3)$. The unique solution $W$ of the system \eqref{eq:lin_prob} accepts the following decomposition
	\begin{equation}
		\label{eq:DecompW} W(\tau,x) = \underline{W}\pare{x_3} + \overline{W}\pare{x} + W_{\osc}\pare{\tau, x},
	\end{equation} 
	where
	\begin{equation*}
		\begin{aligned}
			\underline{W}\pare{x_3} &= \sum_{n_3 \in \mathbb{Z}} \sum_{j \in \set{1,2,3}}  \psca{\widehat{W}_0(0,0,n_3), f_j}_{\mathbb{C}^4} e^{i\cn_3 x_3} f_j\\
			\overline{W}\pare{x} &= \sum_{\substack{n\in\bZ^3\\n_h\neq 0}} \psca{\widehat{W}_0(n), e_0(n)}_{\mathbb{C}^4} e^{i\cn \cdot x} e_0(n)\\
			W_{\osc}\pare{\tau, x} &= \sum_{\substack{n\in\bZ^3\\n_h\neq 0}} \sum_{\alpha \in \set{\pm}} \psca{\widehat{W}_0(n), e_\alpha(n)}_{\mathbb{C}^4} e^{i\tau \omega^\alpha(n)} e^{i\cn \cdot x} e_\alpha(n).
		\end{aligned}
	\end{equation*}
\end{lemma}

\bigskip

Now, we set
\begin{align*}
	&E_\alpha(n,x) = e^{i\cn \cdot x} e_\alpha(n), \hspace{1cm} \forall \, n\in\bZ^3, n_h\neq 0, \forall \, \alpha\in\set{0,\pm}\\
	&F_j(n_3,x_3) = e^{i\cn_3 x_3} f_j, \hspace{1.1cm} \forall \, n_3\in\bZ, \forall \, j\in\set{1,2,3},
\end{align*}
then, $\set{E_\alpha(n,\cdot), F_j(n_3,\cdot)}$ forms an orthonormal basis of $L^2_\sigma(\T^3)$, and we have the following decomposition

\begin{definition}
	\label{def:DecompV} For any vector field $V \in L^2_\sigma(\T^3)$, we have
	\begin{equation}
		\label{eq:DecompV} V(x) = \underline{V}\pare{x_3} + \overline{V}\pare{x} + V_{\osc}\pare{x},
	\end{equation}
	where 
	\begin{equation*}
		\begin{aligned}
			\underline{V}\pare{x_3} &= \sum_{n_3 \in \mathbb{Z}} \sum_{j \in \set{1,2,3}}  \psca{\widehat{V}(0,0,n_3), f_j}_{\mathbb{C}^4} F_j(n_3,x_3)\\
			\overline{V}\pare{x} &= \sum_{\substack{n\in\bZ^3\\n_h\neq 0}} \psca{\widehat{V}(n), e_0(n)}_{\mathbb{C}^4} E_0(n,x)\\
			V_{\osc}\pare{\tau, x} &= \sum_{\substack{n\in\bZ^3\\n_h\neq 0}} \sum_{\alpha \in \set{\pm}} \psca{\widehat{V}(n), e_\alpha(n)}_{\mathbb{C}^4} E_\alpha(n,x).
		\end{aligned}
	\end{equation*}
\end{definition}

\noindent We also have the following result

\begin{prop}
	\label{prop:Projection}
	Let $\Pi_X$ be the projection onto the subspace $X$ of $L^2_\sigma(\T^3)$. For any vector field $V \in L^2_\sigma(\T^3)$, we have
	\begin{enumerate}
		\item \label{KerPA1} $\Pi_{\underline{L^2_\sigma}} V = \underline{V}\pare{x_3}$.
			
		\medskip
			
		\item \label{KerPA2} $\Pi_{\widetilde{L^2_\sigma}} V = \widetilde{V}(x) = V(x) - \underline{V}\pare{x_3} = \overline{V}\pare{x} + V_{\osc}\pare{x}$.
		
		\medskip
		
		\item \label{KerPA3} $\underline{V}\pare{x_3} + \overline{V}\pare{x} = \Pi_{\ker \pare{\PA}} V = \Pi_{\ker \pare{\cL - \textnormal{Id}}} V$.
	\end{enumerate}
	Thus, the operator $\cL(\tau)$ only acts on the oscillating part $V_{\osc}$ of $V$.
\end{prop}

\begin{proof}
	The points \eqref{KerPA1} and \eqref{KerPA2} are immediate consequences of the identity \eqref{eq:bar-tilde_dec} and of Definition \eqref{def:DecompV}. The only non evident point is \eqref{KerPA3}, the proof of which simply follows the lines of the proof of \cite[Proposition 4.1]{grenierrotatingeriodic}.
\end{proof}

\section{Analysis of the filtered equation} \label{sec:filt-sys}

In this section, we will use the method of \cite{schochet}, \cite{grenierrotatingeriodic}, \cite{Gallagher_schochet} or \cite{paicu_rotating_fluids} to filter out the high oscillation term in the system \eqref{PBSe}. In order to do so, we first decompose the initial data in the same way as in \eqref{eq:DecompV}, \emph{i.e.}, we write
\begin{equation*}
	V_0= \underline{V_0}+\widetilde{V}_0 = \underline{V_0}+\overline{V}_0 + V_{\osc, 0}.
\end{equation*}

\bigskip

We recall that in the introduction, we defined $\cL$ as the operator which maps $W_0\in L^2_\sigma(\T^3)$ to the solution $W$ of the linear system \eqref{eq:lin_prob}. Using this operator, we now defined the following auxiliary vector field
\begin{equation*}
	U^\varepsilon = \cL \pare{-\frac{t}{\varepsilon}} V^\varepsilon.
\end{equation*}
Replacing $V^\varepsilon = \cL \pare{\frac{t}{\varepsilon}} U^\varepsilon$ into the initial system \eqref{PBSe}, straightforward computations show that $U^\varepsilon$ satisfies the following ``filtered'' system 
\begin{equation}
 	\tag{S${}_\varepsilon$}
	\label{eq:filt-sys}
	\left\lbrace
	\begin{aligned}
		&\partial_t U^\varepsilon+\mathcal{Q}^\varepsilon \pare{U^\varepsilon,U^\varepsilon}- \cA_2^\varepsilon \pare{ D } U^\varepsilon= 0,\\
		&\bigl. U^\varepsilon\bigr|_{t=0}=V_0,
	\end{aligned}
	\right.
\end{equation}
where
\begin{align}
	\label{eq:def_Qeps} \mathcal{Q}^\varepsilon \left( V_1,V_2\right)= & \frac{1}{2} \Lplus \mathbb{P} \left[ \Lminus V_1 \cdot \nabla \Lminus V_2 + \Lminus V_2 \cdot \nabla \Lminus V_1
	\right],\\
	\label{eq:def_A2eps} \cA^\varepsilon_2 (D) W = & \Lplus \cA_2 (D) \Lminus W.
\end{align}

\bigskip

In this section, we will consider the evolution of \eqref{eq:filt-sys} as the superposition of its projections onto the subspace of horizontal independent (or average) vector fields $\underline{L^2_\sigma}$ and the subspace of horizontal oscillating vector fields $\widetilde{L^2_\sigma}$. Always denoting $\underline{V}$ and $\widetilde{V}$ the projection of $V \in L^2_\sigma$ onto $\underline{L^2_\sigma}$ and $\widetilde{L^2_\sigma}$, we formally decompose \eqref{eq:filt-sys} as sum of two following systems
\begin{equation*}
 	\left\lbrace
	\begin{aligned}
		&\partial_t \underline{U^\varepsilon} + \underline{\mathcal{Q}^\varepsilon} \pare{U^\varepsilon,U^\varepsilon} - \underline{\cA_2^\varepsilon}(D) U^\varepsilon = 0,\\
		&\bigl. \underline{U^\varepsilon}\bigr|_{t=0}=\underline{V_0},
	\end{aligned}
	\right.
\end{equation*}
and
\begin{equation*}
 	\left\lbrace
	\begin{aligned}
		&\partial_t \widetilde{U^\varepsilon} + \widetilde{\mathcal{Q}^\varepsilon} \pare{U^\varepsilon,U^\varepsilon} - \widetilde{\cA_2^\varepsilon}(D) U^\varepsilon = 0,\\
		&\bigl. \widetilde{U^\varepsilon}\bigr|_{t=0}=\widetilde{V_0},
	\end{aligned}
	\right.
\end{equation*}
where, for the sake of the simplicity, we identify
\begin{align*}
	&\underline{\mathcal{Q}^\varepsilon} \pare{U^\varepsilon,U^\varepsilon} \equiv \underline{\mathcal{Q}^\varepsilon \pare{U^\varepsilon,U^\varepsilon}} && \underline{\cA_2^\varepsilon}(D) U^\varepsilon \equiv \underline{\cA_2^\varepsilon \pare{ D } U^\varepsilon},\\
	&\widetilde{\mathcal{Q}^\varepsilon} \pare{U^\varepsilon,U^\varepsilon} \equiv \widetilde{\mathcal{Q}^\varepsilon \pare{U^\varepsilon,U^\varepsilon}} && \widetilde{\cA_2^\varepsilon}(D) U^\varepsilon \equiv \widetilde{\cA_2^\varepsilon \pare{ D } U^\varepsilon}.
\end{align*}
In what follows, we provide explicit formulas of $\underline{\mathcal{Q}^\varepsilon} \pare{U^\varepsilon,U^\varepsilon}$, $\underline{\cA_2^\varepsilon}(D) U^\varepsilon$, $\widetilde{\mathcal{Q}^\varepsilon} \pare{U^\varepsilon,U^\varepsilon}$ and $\widetilde{\cA_2^\varepsilon}(D) U^\varepsilon$, and we will decompose the vectors $ \cQ^\varepsilon\pare{U^\varepsilon, U^\varepsilon} $ and $ \cA^\varepsilon_2\pare{D}U^\varepsilon $ in the $ L^2_\sigma $ basis 
\begin{equation*}
	\set{E_\alpha(n,x), F_j(n_3,x_3)}_{\substack{n\in\bZ^3\\ n_3\in\bZ\\ \alpha=0, \pm \\ j=1,2,3}}, 
\end{equation*}
given in the previous section. To this end, in what follows, we introduce some additional notations. For any vector field $V \in L^2_\sigma$, we set
\begin{equation*}
	V^a (n) = \left(\left. \widehat{V} (n) \right| e_a(n)  \right)_{\mathbb{C}^4} \ e_a (n), \qquad  \forall\, a=0, \pm, \ \forall\, n=\pare{ n_h, n_3 } \in \mathbb{Z}^3, \ n_h\neq 0, 
\end{equation*}
and
\begin{equation*}
	V^j(0,n_3) = \left(\left. \widehat{V} \left( 0,  n_3 \right) \right| f_j  \right)_{\mathbb{C}^4} \ f_j,  \qquad \forall\, j=1,2,3, \ \forall\, n_3 \in \mathbb{Z}.
\end{equation*}
We also define the following quantities in order to shorten as much as possible the forthcoming expressions
\begin{align*}
	\omega^{a,b,c}_{k,m,n} & = \omega^a (k) + \omega^b (m) - \omega^c (n), & a,b,c=0,\pm,\\
	\omega^{a,b}_n & =\omega^a(n)+\omega^b(n), & a,b=0,\pm,\\
	\widetilde{\omega}^{b,c}_{m, n} &= \omega^b\pare{m}-\omega^c(n), & b,c=0, \pm,\\
	\omega^{a,b}_{k,m} & = \omega^a(k)+\omega^b(m), & a, b=0, \pm, 
\end{align*}
where the eigenvalues $\omega^0(\cdot)$ and $\omega^\pm(\cdot)$ are defined in the previous section.

\bigskip

Following the lines of \cite{Gallagher_schochet} or \cite{paicu_rotating_fluids}, we deduce that, for $ c=0, \pm $, the projection of $ \cQ^\varepsilon\pare{V_1, V_2} $ onto the subspace generated by $ E_c(n,\cdot) $, for any $ n\in \bZ^3 $ such that $ n_h\neq 0 $,  is 
\begin{align*}
	\psca{\cQ^\varepsilon\pare{V_1, V_2} \,\Big\vert\, E_c(n,x)}_{\widetilde{L^2_\sigma}} E_c(n,x) & = 
	\psca{\cF\pare{\cQ^\varepsilon\pare{V_1, V_2}}(n) \,\Big\vert\, e_c(n)}_{\bC^4} e^{i\cn\cdot x} \ e_c(n),\\
	& = \sum_{k=1}^2 \psca{\cF\pare{\widetilde{ \cQ}^\varepsilon_k \pare{V_1, V_2}}(n) \,\Big\vert\, e_c(n)}_{\bC^4} e^{i\cn\cdot x} \ e_c(n), \notag
\end{align*}
where using the divergence-free property, we can write the Fourier coefficients of the bilinear forms $\widetilde{ \cQ }^\varepsilon_k$, $k=1,2$, as follows
\begin{equation*}
	\cF \pare{{ \widetilde{ \cQ}^\varepsilon_1 \pare{V_1, V_2} }} (n) = \sum_{\substack{k+m=n \\k_h, m_h\neq 0\\ a,b, c=0, \pm}} e^{i\frac{t}{\varepsilon}\omega^{a,b,c}_{k,m,n}} \psca{\widehat{\bP}(n) \begin{pmatrix} \cn \\ 0 \end{pmatrix} \cdot \bS \pare{V_1^{a}(k)  \otimes V_2^b(m)} \Big| \ e_c(n)}_{\mathbb{C}^4}\; e_c(n)
\end{equation*}
and
\begin{align*}
	\cF \pare{\widetilde{ \cQ}^\varepsilon_2 \pare{V_1, V_2}} (n) &=  \sum_{\substack{\pare{0,k_3} + m=n\\  m_h\neq 0\\b, c=0, \pm\\j=1,2,3}} e^{i\frac{t}{\varepsilon}\widetilde{\omega}^{b,c}_{m,n}} \psca{\widehat{\bP}(n) \begin{pmatrix} \cn \\ 0 \end{pmatrix} \cdot \bS \pare{V_1^{j}(0,k_3)  \otimes V_2^b(m)} \big| e_c(n)}_{\mathbb{C}^4}\; e_c(n)\\
 	&+ \sum_{\substack{\pare{0,k_3} + m=n\\m_h\neq 0\\b, c=0, \pm\\j=1,2,3}} e^{i\frac{t}{\varepsilon}\widetilde{\omega}^{b,c}_{m,n}} \psca{\widehat{\bP}(n) \begin{pmatrix} \cn \\ 0 \end{pmatrix} \cdot \bS \left( V_1^{b} (m)  \otimes V_2^j (0,k_3) \right) \big| \ e_c(n)}_{\mathbb{C}^4}\; e_c(n). \notag
\end{align*}
Here, $\begin{pmatrix} \cn \\ 0 \end{pmatrix}$ stands for the four-component vector ${}^t (\cn_1,\cn_2,\cn_3,0)$ and $\bP$ is the Leray projection, defined in \eqref{eq:newLeray}. In order to shorten the notations, we use $\bS$ for the following symmetry operator
\begin{equation*}
	\bS \pare{V_1 \otimes V_2} = V_1 \otimes V_2 + V_2 \otimes V_1,
\end{equation*}
for any $V_1 = (v_1, V_1^4)$ and $V_2 = (v_2,V_2^4)$. We remark that in the above summation formula there is no bilinear interaction which involves elements of the form $ V^{j_1}\pare{0, k_3}\otimes V^{j_2}\pare{0, m_3} $ since, a priori,  these represent vector fields the horizontal average of which is not zero and hence they do not belong to $\widetilde{L^2_\sigma}$.

\bigskip

The projection of $ \cQ^\varepsilon\pare{V_1, V_2} $ onto the subspace generated by $F_j(n_3,\cdot)$,  for any $j=1,2,3$, for any $ n_3\in \bZ$, can be computed in a similar way and we get
\begin{multline*} 
	\psca{\cQ^\varepsilon\pare{V_1, V_2} \,\big\vert\, F_j(n_3,x_3)}_{\underline{L^2_\sigma}} F_j(n_3,x_3)\\
	\begin{aligned}
		& = \psca{\cF\pare{\cQ^\varepsilon\pare{V_1, V_2}} (0,n_3) \,\big\vert\, f_j}_{\bC^4} e^{i\cn_3 x_3} f_j,\\
		& =  \sum_{\substack{k+m=(0,n_3) \\ k_h, m_h\neq 0\\  a,b \in \set{ 0,\pm} }} e^{i\frac{t}{\varepsilon}\omega^{a,b}_{k,m}} \psca{\widehat{\bP}(0,n_3) \begin{pmatrix} 0 \\ 0 \\ \cn_3 \\ 0 \end{pmatrix} \cdot \bS \left( V_1^{a} (k)  \otimes  V_2^b (m)\right) \Big|\, f_j}_{\mathbb{C}^4}\; e^{i\cn_3 x_3} f_j,\\
		& = \psca{\cF \pare{\underline{\cQ}^\varepsilon\pare{ V_1, V_2 }} (0,n_3) \,\big\vert\, f_j} \ e^{i\cn_3 x_3} f_j.
	\end{aligned}
\end{multline*}
Hence,
\begin{equation*}
	\cQ^\varepsilon \pare{ V_1, V_2 }\pare{ x } = \sum_{\substack{n\in\bZ^3 \\ n_h\neq 0}} \cF \pare{{ \widetilde{ \cQ}^\varepsilon_1 \pare{V_1, V_2} } + \widetilde{ \cQ}^\varepsilon_2 \pare{V_1, V_2}} (n) e^{i\cn\cdot x} + \sum_{n_3\in\bZ} \cF \pare{ \underline{\cQ}^\varepsilon\pare{ V_1, V_2 } } (0,n_3) \ e^{i\cn_3 x_3}.
\end{equation*}

\bigskip

The decomposition of $\cA^\varepsilon_2\pare{D}W$ can also be calculated as in \cite{Gallagher_schochet}. We have
\begin{equation*}
	\begin{aligned}
		&\psca{\cA^\varepsilon_2\pare{D}W \,\big\vert\, E_b(n,x)}_{\widetilde{L^2_\sigma}} E_b(n,x)  = \sum_{ a=0,\pm} e^{i\frac{t}{\varepsilon}\omega^{a,b}_n} \psca{\cF \cA_2 (n) W^a (n) \,\big\vert\, e_b(n)}_{\mathbb{C}^4} \ e^{i\cn\cdot x} e_b(n),\\
		&\psca{\cA^\varepsilon_2\pare{D}W \,\big\vert\, F_j(n_3,x_3)}_{\underline{L^2_\sigma}} F_j(n_3,x_3) = \psca{\cF \cA_2 (0,n_3) {W}^j (0,n_3) \,\big\vert\, f_j}_{\mathbb{C}^4} \ e^{i\cn_3\cdot x_3} f_j.
	\end{aligned}
\end{equation*}
Thus,
\begin{equation*}
	\widetilde{ \cA^\varepsilon_2 }\pare{ D } W = \widetilde{ \cA^\varepsilon_2 }\pare{ D } \widetilde{W} = \sum_{\substack{n\in\bZ^3 \\ n_h\neq 0}} \sum_{ a, b=0,\pm} e^{i\frac{t}{\varepsilon}\omega^{a,b}_n} \psca{\cF \cA_2 (n) W^a (n) \,\Big|\, e_b(n)}_{\mathbb{C}^4} \ e^{i\cn\cdot x} \ e_b(n),
\end{equation*}
and
\begin{equation*}
	\underline{ \cA^\varepsilon_2 }\pare{ D } W = \underline{ \cA^0_2 }\pare{ D_3 } \underline{W} = \sum_{n_3\in\bZ} \sum _{j=1,2,3} \psca{\cF \cA_2 (0,n_3) {W}^j (0,n_3) \,\big|\, f_j}_{\mathbb{C}^4} \ e^{i\cn_3 x_3} \ f_j = \pare{ \nu\partial_3^2 \underline{W^1}, \nu\partial_3^2 \underline{W^2}, 0, 0 }.
\end{equation*}
We remark that the operator $ \underline{\cA_2^\varepsilon} $ does not present oscillations, \emph{i.e.} it is independent of $ \varepsilon $.

\bigskip

Combining all the above calculations, we get the following decomposition of the system \eqref{eq:filt-sys}
\begin{lemma}
	\label{le:decomp_filtered_syst}
	Let $ V_0\in\Hs, \ s>5/2 $. We set
	\begin{equation*}
		\underline{V_0}\pare{ x_3 } = \frac{1}{4\pi^2 a_1 a_2} \int_{\T^2_h} V_0 \pare{ y_h, x_3 } dy_h,
	\end{equation*}
	and 
	\begin{equation*}
		\widetilde{V}_0 = V_0-\underline{V_0}.
	\end{equation*}
	Then, the local solution $U^\varepsilon $ of \eqref{eq:filt-sys} can be written as the superposition
	\begin{equation*}
		U^\varepsilon = \tU^\varepsilon + \underline{U^\varepsilon}, 
	\end{equation*}
	where $ \tU^\varepsilon $ and $ \underline{U^\varepsilon} $ are local solutions of the equations
	\begin{equation}
		\label{eq:eq_Utilde_epsilon}
		\left\lbrace
		\begin{aligned}
			& \partial_t \tU^\varepsilon + \widetilde{ \cQ }^\varepsilon_1 \pare{ \tU^\varepsilon, \tU^\varepsilon } + \widetilde{ \cQ }^\varepsilon_2 \pare{ U^\varepsilon, U^\varepsilon } - \widetilde{ \cA^\varepsilon_2 }\pare{ D } \tU^\varepsilon=0,\\
			& \dive \tU^\varepsilon =0,\\
			& \left. \tU^\varepsilon\right|_{t=0}=\widetilde{V}_0,
		\end{aligned}
		\right.
	\end{equation}
	and
	\begin{equation}
		\label{eq:eq_Uunderline_epsilon}
		\left\lbrace
		\begin{aligned}
			& \partial_t \underline{U^\varepsilon}+ \underline{\cQ}^\varepsilon \pare{ \tU^\varepsilon, \tU^\varepsilon }- \underline{\cA^0_2}\pare{ D_3 } \underline{U^\varepsilon}=0, \\
			& \left. \underline{U^\varepsilon}\right|_{t=0} = \underline{V_0}.
		\end{aligned}
		\right.
	\end{equation}
\end{lemma}

\section{The limit system} \label{se:lim_syst}

In this short section, the convergence of suitable subsequences of local strong solutions  $\pare{U^\varepsilon}_{\varepsilon > 0}$ of the system \eqref{eq:filt-sys} will be put in evidence. Moreover, using the similar methods as in \cite{Gallagher_schochet} or \cite{Scrobo_Froude_periodic}, we can determine the systems which describe the evolution of such limits. The explicit formulation of these limiting systems will be given in the next section.
We first introduce the limit forms $\widetilde{\cQ}_1$, $\widetilde{\cQ}_2$, $\underline{\cQ}$, $\widetilde{\cA^0_2}$ and $\underline{\cA^0_2}$ such that
\begin{equation}
	\label{eq:limit_cQ1} \widetilde{\mathcal{Q}}_1 (V_1, V_2) = \sum_{\substack{n\in\bZ^3 \\ n_h\neq 0}} \sum_{\substack{\omega^{a,b,c}_{k,m,n}=0\\k+m=n \\ k_h, m_h, n_h\neq 0 \\ a,b,c \in \set{ 0,\pm}}} \psca{\widehat{\mathbb{P}}(n) \begin{pmatrix} n \\ 0 \end{pmatrix} \cdot \bS \pare{V_1^{a} (k)  \otimes  V_2^b (m)} \,\Big|\, e_c(n)}_{\mathbb{C}^4}\; e^{i\cn\cdot x} \ e_c(n), 
\end{equation}
\begin{align}
	\label{eq:limit_cQ2} \widetilde{ \cQ}_2 \pare{V_1, V_2} &= \sum_{\substack{n\in\bZ^3 \\ n_h\neq 0}} \sum_{\substack{\pare{0,k_3} + m=n\\ m_h, n_h \neq 0\\ \widetilde{\omega}^{b,c}_{m,n}=0\\b, c \in \set{0,\pm}\\j=1,2,3}} \psca{\widehat{\mathbb{P}}(n) \begin{pmatrix} n \\ 0 \end{pmatrix} \cdot \bS \pare{V_1^j (0,k_3)  \otimes  V_2^b (m)} \Big| e_c(n)}_{\mathbb{C}^4}\; e^{i\cn\cdot x} \ e_c(n) \\
	& \ + \sum_{\substack{n\in\bZ^3 \\ n_h\neq 0}} \sum_{\substack{\pare{0,k_3} + m=n\\m_h, n_h \neq 0\\ \widetilde{\omega}^{b,c}_{m,n}=0\\b, c \in \set{0,\pm}\\j=1,2,3}} \psca{\widehat{\mathbb{P}}(n) \begin{pmatrix} n \\ 0 \end{pmatrix} \cdot \bS \left( V_1^{b} (m)  \otimes  {V}_2^j (0,k_3)\right) \,\Big|\, e_c(n)}_{\mathbb{C}^4}\; e^{i\cn\cdot x} \ e_c(n), \notag
\end{align}
\begin{equation}
	\label{eq:limit_cQuline} \underline{\cQ} \pare{ V_1, V_2 } = \sum_{n_3\in\bZ} \sum_{\substack{k+m=(0,n_3) \\ k_h, m_h \neq 0\\ \omega^a (k)+\omega^b (m)=0 \\ a,b \in \set{ 0,\pm}\\ j=1,2,3}} \psca{\widehat{\bP}(0,n_3) \begin{pmatrix} 0 \\ 0 \\ \cn_3 \\ 0 \end{pmatrix} \cdot \bS \pare{V_1^a (k)  \otimes V_2^b (m)} \,\Big|\, f_j}_{\mathbb{C}^4}\; e^{i\cn_3 x_3} \ f_j,
\end{equation}
\begin{equation}
	\label{eq:limit_At} \widetilde{\cA^0_2} (D) W = \sum_{\substack{n\in\bZ^3 \\ n_h\neq 0}} \sum_{\substack{\omega^{a,b}_n=0\\ a,b\in \set{0,\pm}}} \psca{\cF \cA_2 (n) \widetilde{W}^a (n) \,\big|\, e_b(n)}_{\mathbb{C}^4} e^{i\cn\cdot x} \ e_b(n),
\end{equation}
and
\begin{equation}
	\label{eq:limit_Au} \underline{ \cA^0_2} (D) W = \sum_{n_3\in\bZ} \sum_{j=1,2,3} \ps{\cF \cA(0,n_3) W^j(0,n_3)}{f_j}_{\bC^4} e^{i\cn_3 x_3} \ f_j. 
\end{equation}
Using the same standard method of the non-stationary phases as in \cite{Gallagher_schochet}, \cite{grenierrotatingeriodic}, \cite{paicu_rotating_fluids} or \cite{Scrobo_primitive_horizontal_viscosity_periodic}, we can prove the following convergence result
\begin{lemma}
	\label{le:lim_smooth}
	Let $ V_1, V_2 $ and $ W $ be zero-average smooth vector fields. Then, we have the following convergence in the sense of distributions
	\begin{align*}
		\widetilde{\cQ}_k^\varepsilon \left( V_1, V_2 \right) \xrightarrow{\varepsilon\to 0} & \ \widetilde{\cQ}_k \left( V_1, V_2 \right), \qquad k=1,2 \\
		\underline{\cQ}^\varepsilon \left( V_1, V_2 \right) \xrightarrow{\varepsilon\to 0} & \ \underline{\cQ} \left( V_1, V_2 \right), \\
		\widetilde{\cA^\varepsilon_2} (D) W \xrightarrow{\varepsilon\to 0} & \ \widetilde{\cA^0_2} (D) W,\\
		\underline{\cA^\varepsilon_2} \left( D_3 \right) W \xrightarrow{\varepsilon\to 0} & \ \underline{\cA^0_2} \left( D_3 \right) W.
	\end{align*}
\end{lemma}

Let us now define the (limit) bilinear forms
\begin{equation}
	\label{eq:cQ_cA}
	\begin{aligned}
		\cQ \pare{V_1, V_2} & = \widetilde{\cQ}_1 \pare{V_1, V_2} + \widetilde{\cQ}_2 \pare{V_1, V_2} + \underline{\cQ} \pare{V_1, V_2},\\
		\cA^0_2 (D) W  & = \widetilde{\cA^0_2} (D) W + \underline{\cA^0_2} (D) W.
 	\end{aligned}
\end{equation}
Then, the limit dynamics of \eqref{PBSe} as $\varepsilon\to 0$ can be described as follows
\begin{prop} 
	\label{pr:compact_strong} Let $V_0 \in \Hs$, $s>5/2$ and $T\in \left[0,T^\star \right[$, where $T^\star > 0$ is defined in Theorem \ref{thm:local}. The sequence $\left( U^\varepsilon \right)_{\varepsilon>0}$ of local strong solutions of \eqref{eq:filt-sys}, which is uniformly bounded in $ \cC \left( \left[0,T\right]; \Hs \right) $, is compact in the space $ \cC \left( \left[0,T\right]; H^{\sigma} \left( \T^3 \right) \right) $ where $ \sigma\in \left( s-2, s \right) $. Moreover each limit point $U$ of $ \left( U^\varepsilon \right)_{\varepsilon>0}$ solves the following limit equation
	\begin{equation}
		\label{eq:limit_system}
		\tag{{S}$_0 $}
		\left\lbrace
		\begin{aligned}
			&\partial_t  U + \mathcal{Q} \left(  U,  U \right) - \cA^0_2 (D)  U=0,\\
			& \left.  U \right|_{t=0}=  V_0,
		\end{aligned}
		\right.
	\end{equation}
	\emph{a.e.} in $\T^3\times [0,T]$, where $\mathcal{Q}$ and $\cA^0_2 (D)$ are given in \eqref{eq:cQ_cA}.
\end{prop}

\begin{proof}
The proof of the compactness of the sequence $ \left( U^\varepsilon \right)_{\varepsilon>0} $ is a  standard argument. Thanks to Theorem  \ref{thm:local} we know that $ \left( U^\varepsilon \right)_{\varepsilon>0} $ is uniformly bounded in $\cC \left( \left[0,T\right]; \Hs \right)$. Since
\begin{equation*}
	\left\| \cQ^\varepsilon \left( U^\varepsilon, U^\varepsilon \right) \right\|_{H^{s-1}} \leqslant C \left\| U^\varepsilon \right\|_{H^s}^2,
\end{equation*}
and
\begin{equation*}
	\left\| A^\varepsilon_2 (D) U^\varepsilon \right\|_{H^{s-2}} \leqslant C \left\| U^\varepsilon \right\|_{H^s},
\end{equation*}
we deduce that $\left( \partial_t U^\varepsilon \right)_{\varepsilon>0}$ is uniformly bounded in $\cC \left( \left[0,T\right]; H^{s-2} \right)$. Remarking that
\begin{equation*}
	\cC \left( \left[0,T\right]; H^{s-2} \right) \hra L^1 \left( \left[0,T\right]; H^{s-2} \right),
\end{equation*}
we can hence apply Aubin-Lions lemma \cite{Aubin63} to deduce the compactness of the sequence.

Finally, we need to prove that a limit point $U$ (weakly) solves the limit system \eqref{eq:limit_system}. We remark that Lemma \ref{le:lim_smooth} cannot be directly applied since the sequence $\left( U^\varepsilon \right)_{\varepsilon >0}$ is not sufficiently regular. However, by mollifying the data as in \cite{grenierrotatingeriodic}, \cite{schochet} or \cite{Scrobo_primitive_horizontal_viscosity_periodic}, we can deduce that $U$ solves \eqref{eq:limit_system} in $\cD'\pare{\T^3\times [0, T]}$. We choose now $\overline{\sigma}\in \pare{\frac{5}{2}, s}$. Since $ U\in \cC\pare{[0, T]; H^{\overline{\sigma}}}$, Sobolev embeddings imply that $U\in \cC\pare{[0, T]; \cC^{1, 1}}$, and so it solves the system \eqref{eq:limit_system} \emph{a.e.}  in $\T^3\times [0,T]$
\end{proof}

\begin{rem}
	A similar result to Proposition \ref{pr:compact_strong} was proved in \cite[Lemma 3.3]{Scrobo_Froude_periodic}. We underline though that the proof of \cite[Lemma 3.3]{Scrobo_Froude_periodic} cannot be adapted to the present setting since it strongly used the regularity induced by the uniform parabolic smoothing effects.  
\end{rem}

\section{Explicit formulations of the limit system} \label{sec:decomposition_limit}

In this section, we determine the explicit formulation of the limit system \eqref{eq:limit_system}. More precisely, for each limit point $U$ of the sequence of solutions $\left( U^\varepsilon \right)_{\varepsilon>0}$ of \eqref{eq:filt-sys}, we decompose the bilinear term $\cQ(U,U)$ and the linear term $\cA_2^0(D)U$ by mean of the projections onto $\ker \PA$ defined in Proposition \ref{prop:Projection} and onto its orthogonal $\pare{\ker \PA}^\perp$. Such approach is generally used for hyperbolic symmetric systems with skew-symmetric perturbation in periodic domains, and we refer the reader to \cite{Gallagher_schochet}, \cite{Gallagher_singular_hyperbolic}, \cite{Gallagher_incompressible_limit}, \cite{paicu_rotating_fluids}, \cite{Scrobo_Froude_periodic} or \cite{Scrobo_primitive_horizontal_viscosity_periodic},  for instance, for some other related systems. Since the spectral properties of the operator $ \PA $ are rather different in the Fourier frequency subspaces $\set{ n_h=0 }$ and $\set{ n_h \neq 0 }$, as in Section \ref{sec:filt-sys}, we write the limit system \eqref{eq:limit_system} as the superposition of the following systems
\begin{equation}
	\label{eq:lim_syst_Uunderline}
	\left\lbrace
	\begin{aligned}
		&\partial_t \uU + \underline{\cQ} \pare{\tU, \tU} - \underline{\cA^0_2}\pare{D_3}\uU =0,\\
		&\uU^3\equiv 0,\\
		&\left. \uU\right|_{t=0}=\underline{V_0},
	\end{aligned}
	\right.
\end{equation}
and
\begin{equation}
	\label{eq:lim_syst_Utilde}
	\left\lbrace
	\begin{aligned}
		&\partial_t \tU + \widetilde{\cQ}_1 \left( \tU, \tU \right) +\widetilde{\cQ}_2 \left( U, U \right)- \widetilde{\cA^0_2}(D)\tU=0,\\
		&\dive \tU=0,\\
		&\left. \tU\right|_{t=0}=\widetilde{V}_0,
	\end{aligned}
	\right.
\end{equation}
where the limit terms $\underline{\cQ} \pare{\tU, \tU}$, $\underline{\cA^0_2}\pare{D_3}\uU$, 
$\widetilde{\cQ}_1 (\tU, \tU)$, $\widetilde{\cQ}_2 (\tU, \tU)$ and $\widetilde{\cA^0_2}(D)\tU$ are defined as in Lemma \ref{le:lim_smooth}. In what follows, we will separately study the systems \eqref{eq:lim_syst_Uunderline} and \eqref{eq:lim_syst_Utilde}.

\subsection{The dynamics of $\uU$}

We will prove the following result, which allows to simplify the system \eqref{eq:lim_syst_Uunderline}.
\begin{prop} \label{prop:limit_eq_Uunderline_local}
	Let ${V_0}\in \Hs, s>5/2$, and $V_0 = \underline{V_0} + \widetilde{V_0}$ as in Lemma \ref{le:decomp_filtered_syst}. Then, the horizontal average $\uU=\pare{\uu, 0, \underline{U^4}}$ of the solution $U$ of the limit system \eqref{eq:limit_system} solves the  homogeneous diffusion system \eqref{eq:Uunder_heat2D} a.e. in $ \left[0,T\right] \times \T^1_{\v} $ for each $T\in \left[0, T^\star \right[$.
\end{prop}
\noindent In order to prove Proposition \ref{prop:limit_eq_Uunderline_local}, we only need to prove the following lemma
\begin{lemma}
	\label{lem:zero_limit_bilinear_hor_average}
	The following identity holds true
	\begin{equation*}
 		\underline{\cQ} \pare{\tU, \tU}=0.
	\end{equation*}
\end{lemma}

We recall that in Section \ref{sec:filt-sys}, we already introduce, for any vector field $V \in L^2_\sigma$, 
\begin{equation*}
	V^a (n) = \left(\left. \widehat{V} (n) \right| e_a(n)  \right)_{\mathbb{C}^4} \ e_a (n), \qquad  \forall\, a=0, \pm, \ \forall\, n=\pare{ n_h, n_3 } \in \mathbb{Z}^3, \ n_h\neq 0.
\end{equation*}
Here, we also denote
\begin{equation*}
	\widehat{V}^a (n)  = \left(\left. \widehat{V} (n) \right| e_a (n)  \right)_{\mathbb{C}^4},
\end{equation*}
and $V^{a,l}$, $l=1,2,3,4$ the $l$-th component and respectively $V^{a,h}$ the first two components of the vector $V^a$. Using the definition of $\underline{\cQ}$ and the particular form of the vector ${}^t(0,0,\cn_3,0)$ given in \eqref{eq:limit_cQuline}, we have
\begin{align*}
	\cF \pare{\underline{\cQ} \left( U, U \right)}(0,n_3) & = \sum_{\substack{k+m=(0,n_3) \\ k_h, m_h \neq 0\\ \omega^a (k)+\omega^b(m)=0 \\ a, b =0, \pm \\ j=1, 2, 3}} \psca{\widehat{\bP}(0,n_3) \, n_3 \, U^{a,3}(k) \, U^b (m) \,\Big\vert\, f_j}_{\mathbb{C}^4} f_j,\\
	& = \sum_{a,b=0,\pm} \sum_{j=1}^3 \sum_{\mathcal{I}_{a,b} (n_3)} \psca{\widehat{\bP}(0,n_3) \, n_3 \, U^{a, 3}(k) \, U^b(m) \,\Big\vert\, f_j}_{\mathbb{C}^4} f_j,
\end{align*}
where the set $\mathcal{I}_{a, b}(n_3)$ contains the following resonance frequencies
\begin{equation}
	\label{eq:Iab} \mathcal{I}_{a,b} (n_3) = \set{  (k,m)\in \mathbb{Z}^6, \ k_h, m_h\neq 0 \; \left| \; k+m= \left( 0,n_3 \right),   \; \omega^a (k) + \omega^b (m)=0\Big. \right.}.
\end{equation}
In order to prove Lemma \ref{lem:zero_limit_bilinear_hor_average}, we will show that, for each couple $\pare{ a, b }\in \set{ 0, \pm }^2$, the contribution
\begin{equation}
	\label{eq:Jab} \mathcal{J}_{a,b}(n_3) = \sum_{\mathcal{I}_{a,b} (n_3)} \widehat{\bP}(0,n_3) \, n_3 \, U^{a, 3}(k) \, U^b(m),
\end{equation}
is null, which implies that
\begin{equation}
	\label{eq:Q_come_Jab}
	\cF \pare{\underline{\cQ} \left( U, U \right)}(0,n_3) =  \sum_{a,b=0,\pm} \sum_{j=1}^3 \psca{\mathcal{J}_{a,b}(n_3) \,\Big\vert\, f_j}_{\mathbb{C}^4} f_j=0.
\end{equation}

\bigskip

\noindent \underline{\textit{Case 1: $(a,b)=(0,0)$}}. We have
\begin{equation*}
	\cJ_{0,0} (n_3)= \sum_{k+m = \left( 0,n_3 \right)} \widehat{\bP}(0,n_3) \, n_3 \, U^{0,3}(k) \, U^{0}(m) = 0,
\end{equation*}
since $ U^{0,3}\equiv 0 $ (see \eqref{eq:eigenvectors} and \eqref{ev_nh=0}).

\bigskip

\noindent \underline{\textit{Case 2: $(a,b)=(\pm,0)$ or $(a,b)=(0,\pm)$}}. If $(a,b)=(\pm,0)$, then,
\begin{equation*}
	\mathcal{J}_{\pm, 0} (n_3) = \sum_{\substack{ k+m = (0,n_3) \\ \omega^\pm (k)=0 }} \widehat{\bP}(0,n_3) \, n_3 \, U^{\pm,3}(k)   \, U^0(m).
\end{equation*}
The condition $ \omega^\pm (k)=0 $ implies that $ k_h \equiv 0 $, while the condition $ k+m = (0,n_3) $ implies that $ m_h \equiv 0 $. Then from \eqref{ev_nh=0}, we have $ U^{a, 3} (0,k_3)\equiv 0 $, which shows that $\mathcal{J}_{\pm, 0} (n_3)$ gives a null contribution in \eqref{eq:Q_come_Jab}. The same approach can be applied to the case $ \left( a, b \right)= \left( 0, \pm \right) $.

\bigskip

\noindent \underline{\textit{Case 3: $(a,b)=(+,+)$ or $(a,b) = (-,-)$}}. In this case, $\mathcal{J}_{a,b}(n_3)$ writes
\begin{equation*}
	\cJ_{\pm, \pm} (n_3) = \sum_{\substack{ k+m = (0,n_3) \\ \omega^\pm (k) + \omega^\pm (m)=0 }}  \widehat{\bP}(0,n_3) \, n_3 \, U^{\pm,3}(k) \, U^\pm(m).
\end{equation*}
Since $k+m = (0,n_3)$, we can set $\left| \ck_h \right|= \left| \cm_h \right|=\lambda$. We deduce from the constraint $ \omega^\pm (k) + \omega^\pm (m)=0 $ and the explicit formulation of the eigenvalues \eqref{eq:eigenvalues} that
\begin{equation*}
	\frac{\lambda}{\sqrt{\lambda^2 + \ck_3^2}} + \frac{\lambda}{\sqrt{\lambda^2 + \cm_3^2}}=0,
\end{equation*}
which implies that $ \lambda\equiv 0 $. Then the similar argument as in Cases 1 and 2 shows that $\cJ_{\pm, \pm} (n_3) = 0$.

\bigskip

\noindent \underline{\textit{Case 4: $(a,b)=(+,-)$ or $(a,b) = (-,+)$}}. This is the most delicate case to treat. We write
\begin{equation}
	\label{eq:bil_contr_pm_mp}
	\cJ_{\pm, \mp} (n_3)= \sum_{\substack{ k+m = (0,n_3) \\ \omega^\pm (k) = \omega^\pm (m) }} \widehat{\bP}(0,n_3) \, n_3 \, U^{\pm,3}(k)  \, U^{\mp}(m).
\end{equation}
The conditions $k+m = (0,n_3)$ and $\omega^\pm (k) = \omega^\pm (m)$ imply now that $k_h = -m_h$, $\left| \ck_h \right|= \left| \cm_h \right|=\lambda$ and
\begin{equation*}
	\frac{\lambda}{\sqrt{\lambda^2 + \ck_3^2}} = \frac{\lambda}{\sqrt{\lambda^2 + \cm_3^2}}.
\end{equation*}
Then, it is obvious that $ m_3=\pm k_3 $.

If $ m_3=-k_3 $ then the convolution constraint $ k_3+m_3=n_3 $ in \eqref{eq:bil_contr_pm_mp} implies that $ n_3\equiv 0 $, and hence there is no contributions of $\cJ_{\pm, \mp} (n_3)$ in \eqref{eq:Q_come_Jab}. So, we concentrate on the case where $ k_h = -m_h $ and $ k_3=m_3= \frac{n_3}{2}$ and we will deal with the interaction will be of the form, for any $n_3 \in 2 \mathbb{Z}$,
\begin{equation*}
	\cJ_{\pm, \mp} (n_3)= \sum_{m_h \in \mathbb{Z}^2} \widehat{\bP}(0,n_3) \, n_3 \, U^{\pm,3} \pare{-m_h , \frac{n_3}{2}}  \, {U}^{\mp} \pare{ m_h, \frac{n_3}{2} }.
\end{equation*}
For any $n_3 \in 2 \mathbb{Z}$, we set
\begin{align*}
	B^{\pm, \mp}_{n_3} & = \sum_{m_h \in \mathbb{Z}^2} n_3 \, {U}^{\pm,3} \left( -m_h , \frac{n_3}{2} \right)  \; {U}^{\mp, h} \left( m_h, \frac{n_3}{2} \right) , \\
	C^{\pm, \mp}_{n_3} & = \sum_{m_h \in \mathbb{Z}^2} n_3 \, {U}^{\pm,3} \left( -m_h , \frac{n_3}{2} \right)  \; {U}^{\mp, 4} \left( m_h, \frac{n_3}{2} \right).
\end{align*}
Taking into account the form of the vectors $f_j$ in \eqref{ev_nh=0}, we deduce that
\begin{equation*}
	\sum_{j=1}^3 \psca{\cJ_{+,-} (n_3) +\cJ_{-,+} (n_3) \, \vert \, f_j}_{\bC^4} f_j = \widehat{\bP}(0,n_3)  
	\begin{pmatrix}
		B^{+,-}_{n_3} + B^{-, +}_{n_3}\\
		0\\
		C^{+,-}_{n_3} + C^{-, +}_{n_3}
	\end{pmatrix}.
\end{equation*}
Then, we can prove that the sum $\cJ_{+,-} (n_3) +\cJ_{-,+} (n_3)$ have no contribution in \eqref{eq:Q_come_Jab} and conclude Case 4 if we prove the following lemma
\begin{lemma}
	\label{lem:BCpmmp} For any $n_3 \in 2 \mathbb{Z}$, we have the following identities
	\begin{equation}
		\label{eq:Bpmmp=0} B^{+,-}_{n_3} = -B^{-,+}_{n_3}, 
	\end{equation}
	and
	\begin{equation}
		\label{eq:Cpmmp=0} C^{+,-}_{n_3} = -C^{-,+}_{n_3}. 
	\end{equation}
\end{lemma}

\bigskip

\begin{proof}

Identity \eqref{eq:Cpmmp=0} is quite easy to prove. Indeed, we have
\begin{equation*}
	C^{\pm, \mp}_{n_3} = \sum_{m_h\in\bZ^2} n_3 \ e_\pm^3 \pare{ -m_h, \frac{n_3}{2} }  \widehat{U}^{\pm} \pare{ -m_h, \frac{n_3}{2} } \widehat{U}^{\mp} \pare{-m_h, \frac{n_3}{2} } .
\end{equation*}
Then,
\begin{multline}
	\label{eq:Cpm+Cmp}
	C^{+,-}_{n_3} + C^{-,+}_{n_3}= \sum_{m_h\in \bZ^2} \frac{n_3}{2}  \left[ e_\pm^3 \pare{ -m_h, \frac{n_3}{2} }  \widehat{U}^{\pm} \pare{ -m_h, \frac{n_3}{2} } \widehat{U}^{\mp} \pare{-m_h, \frac{n_3}{2}}\right.\\
	+ \left. \ e_\mp^3 \pare{ m_h, \frac{n_3}{2} }  \widehat{U}^{\mp} \pare{ m_h, \frac{n_3}{2} } \widehat{U}^{\pm} \pare{-m_h, \frac{n_3}{2}}\right].
\end{multline}

\noindent
The explicit expression of $e_\mp(n)$ in \eqref{eq:eigenvectors} yields
\begin{equation}
	\label{eq:bof1} e_\pm^3 \pare{ -m_h, \frac{n_3}{2} } = - e_\mp^3 \pare{ m_h, \frac{n_3}{2} },
\end{equation}
which, combined with \eqref{eq:Cpm+Cmp}, implies \eqref{eq:Cpmmp=0}.

\medskip

To prove Identity \eqref{eq:Bpmmp=0}, we consider the quantities
\begin{multline}
	\label{def_beta} \beta \left( m_h, n_3 \right) = \frac{n_3}{2} \left[ {U}^{+,3} \left( -m_h , \frac{n_3}{2} \right)  \; {U}^{-, h} \left( m_h, \frac{n_3}{2} \right) + {U}^{-,3} \left( -m_h , \frac{n_3}{2} \right)  \; {U}^{+, h} \left( m_h, \frac{n_3}{2} \right)\right.\\
	+ \left. {U}^{+,3} \left( m_h , \frac{n_3}{2} \right)  \; {U}^{-, h} \left( -m_h, \frac{n_3}{2} \right) + {U}^{-,3} \left( m_h , \frac{n_3}{2} \right)  \; {U}^{+, h} \left( -m_h, \frac{n_3}{2} \right) \right],
\end{multline}
which allows to write
\begin{equation}
	\label{Bpm} B^{+,-}_{n_3}+ B^{-,+}_{n_3} = \sum _{m_h \in \mathbb{Z}^2}  \beta \left( m_h, n_3 \right).
\end{equation}
Now, we decompose
\begin{equation*}
	\beta \left( m_h, n_3 \right)= \beta^+ \left( m_h, n_3 \right)+ \beta^- \left( m_h, n_3 \right), 
\end{equation*}
where
\begin{align*}
	\beta^+ \left( m_h, n_3 \right)= & \ \frac{n_3}{2} \left[ {U}^{+,3} \left( -m_h , \frac{n_3}{2} \right)  \; {U}^{-, h} \left( m_h, \frac{n_3}{2} \right) + {U}^{-,3} \left( m_h , \frac{n_3}{2} \right)  \; {U}^{+, h} \left( -m_h, \frac{n_3}{2} \right)  \right],\\
	 \beta^- \left( m_h, n_3 \right)= & \ \frac{n_3}{2} \left[ {U}^{-,3} \left( -m_h , \frac{n_3}{2} \right)  \; {U}^{+, h} \left( m_h, \frac{n_3}{2} \right) + {U}^{+,3} \left( m_h , \frac{n_3}{2} \right)  \; {U}^{-, h} \left( -m_h, \frac{n_3}{2} \right) \right].
\end{align*}
By definition, we have
\begin{multline*}
	{U}^{+,3} \left( -m_h , \frac{n_3}{2} \right)  \; {U}^{-, h} \left( m_h, \frac{n_3}{2} \right) = \psca{\widehat{U} \left( -m_h, \frac{n_3}{2} \right) \,\Big\vert\, e_+ \left( -m_h, \frac{n_3}{2} \right) }_{\mathbb{C}^4} e_+^3 \left( -m_h, \frac{n_3}{2} \right)\\
	\times \psca{\widehat{U} \left( m_h, \frac{n_3}{2} \right) \,\Big\vert\, e_- \left( m_h, \frac{n_3}{2} \right) }_{\mathbb{C}^4} e_-^h \left( m_h, \frac{n_3}{2} \right),
\end{multline*}
and
\begin{multline*}
	{U}^{-,3} \left( m_h , \frac{n_3}{2} \right)  \; {U}^{+, h} \left( -m_h, \frac{n_3}{2} \right) = \psca{ \widehat{U} \left( m_h, \frac{n_3}{2} \right) \,\Big\vert\, e_- \left( m_h, \frac{n_3}{2} \right) }_{\mathbb{C}^4} e_-^3 \left( m_h, \frac{n_3}{2} \right)\\
	\times \psca{ \widehat{U} \left( -m_h, \frac{n_3}{2} \right) \,\Big\vert\, e_+ \left( -m_h, \frac{n_3}{2} \right) }_{\mathbb{C}^4} e_+^h \left( -m_h, \frac{n_3}{2} \right).
\end{multline*}
The explicit formula \eqref{eq:eigenvectors} implies
\begin{align*}
 	e_-^h \left( m_h, \frac{n_3}{2} \right) = & \ e_+^h \left( -m_h, \frac{n_3}{2} \right) = A^h_{m_h, n_3}  ,\\
 	e_+^3 \left( -m_h, \frac{n_3}{2} \right) = & \ - e_-^3 \left( m_h, \frac{n_3}{2} \right) = A^3_{m_h, n_3} .
\end{align*}
Setting
\begin{equation*}
	C_{m_h, n_3}= \psca{ \widehat{U} \left( -m_h, \frac{n_3}{2} \right) \,\Big\vert\, e_+ \left( -m_h, \frac{n_3}{2} \right) }_{\mathbb{C}^4} \ \psca{ \widehat{U} \left( m_h, \frac{n_3}{2} \right) \,\Big\vert\, e_- \left( m_h, \frac{n_3}{2} \right) }_{\mathbb{C}^4},
\end{equation*}
we obtain
\begin{align*}
{U}^{+,3} \left( -m_h , \frac{n_3}{2} \right)  \; \widehat{U}^{-, h} \left( m_h, \frac{n_3}{2} \right) = & \ - C_{m_h, n_3} A^h_{m_h, n_3} A^3_{m_h, n_3},\\
{U}^{-,3} \left( m_h , \frac{n_3}{2} \right)  \; \widehat{U}^{+, h} \left( -m_h, \frac{n_3}{2} \right) = & \ C_{m_h, n_3} A^h_{m_h, n_3} A^3_{m_h, n_3},
\end{align*}
which imply
\begin{equation*}
	\beta^+ \left( m_h, n_3 \right) \equiv 0.
\end{equation*}
By the similar argument, we also get
\begin{equation*}
	\beta^- \left( m_h, n_3 \right) \equiv 0,
\end{equation*}
which yields
\begin{equation*}
	\beta \left( m_h, n_3 \right) \equiv 0.
\end{equation*}
Thus, Identity \eqref{Bpm} implies \eqref{eq:Bpmmp=0}

\end{proof}

\subsection{The dynamics of $ \tU $}

In the previous paragraph, the dynamics of $\bU$ is well understood and turns out to follow quite a simple heat equation. To complete the study of the limit system \eqref{eq:limit_system}, we now give an explicit expression of the system \eqref{eq:lim_syst_Utilde}. As in Lemma \ref{lem:DecompW} and Proposition \ref{prop:Projection} \eqref{KerPA2}, we will study the evolution of $ \tU $ as the superposition of
\begin{equation*}
	\tU=\bU+U_{\osc}.
\end{equation*}
The main technical difficulty of the study consists in giving a close formulation of the projection of the bilinear interactions, which was considered in \cite{Scrobo_Froude_periodic}. In what follows, we only mention the main steps of the study, without going into technical calculations.

\subsubsection{Derivation of the evolution of $ \bU $}

We recall that $ \bU $ is the projection of $ \tU $ onto the nonoscillating subspace generated by $\set{E_0(n,\cdot)}_n$. The derivation of the evolution of $ \bU $ can be done in three steps.

\bigskip

\noindent \underline{\textit{Step 1}}: We explicitly compute the projections of $ \widetilde{\cQ}_1 \pare{\tU, \tU} $ and $ \widetilde{\cQ}_2 \pare{U, U} $ onto $\text{Span}\set{E_0(n,\cdot)}$, \emph{i.e.}, for any $n\in\bZ^3$, $n_h \neq 0$, we compute the following quantities
\begin{align*}
 	\overline{\widetilde{\cQ}_1 \pare{\tU, \tU}} &= \sum_{\substack{n\in\bZ^3 \\ n_h\neq 0}} \ps{ \cF \widetilde{\cQ}_1 \pare{\tU, \tU}(n)}{e_0(n)}_{\bC^4} e^{i\cn\cdot x} \ e_0(n)\\
 	\overline{\widetilde{\cQ}_2 \pare{U, U}} &= \sum_{\substack{n\in\bZ^3 \\ n_h\neq 0}} \ps{ \cF \widetilde{\cQ}_2 \pare{U, U}(n)}{e_0(n)}_{\bC^4} e^{i\cn\cdot x} \ e_0(n).
\end{align*}

The projection of $ \widetilde{\cQ}_1 \left( \tU, \tU \right) $ is a mere horizontal transport interaction of elements in the kernel of $ \PA $ as it is showed in the following lemma, the proof of which can be found in \cite[Lemma 4.2]{Scrobo_Froude_periodic}.
\begin{lemma} \label{lem:proj_tQ1_e0}
 	The following identity holds true
 	\begin{equation*}
 		\overline{\widetilde{\cQ}_1 \pare{\tU, \tU}} = \begin{pmatrix} \uh\cdot\nh \uh \\ 0\\ 0 \end{pmatrix} + \begin{pmatrix} \nh \overline{p}_1 \\ 0\\ 0 \end{pmatrix},
 	\end{equation*}
 	where
 	\begin{align*}
 		\uh	 & = \nhp \left( -\Dh \right)^{-1} \left( -\partial_2 U^1 + \partial_1 U^2 \right), \\
 		\overline{p}_1 & =  \left( -\Dh \right)^{-1} \diveh \diveh \left( \uh \otimes \uh \right).
 	\end{align*}
\end{lemma}

For the projection of $  \widetilde{\cQ}_2 \left( U, U \right) $, we remark that the matrix $\widehat{\bP}(n)$ real and symmetric, so we can write
\begin{equation*}
	\cF \widetilde{\cQ}_2 \left( U, U \right) (n) = \cF\widetilde{\cQ}_2 \left( \uU, \tU \right) (n) = 2 \sum_{\substack{\pare{0,k_3} + m=n\\ m_h, n_h \neq 0 \\ \widetilde{\omega}^{b,c}_{m,n}=0\\b, c =0, \pm\\j=1,2,3}} \psca{\widehat{\bP}(n) \begin{pmatrix} \cn \\ 0 \end{pmatrix} \cdot \bS \left( {U}^{j} (0,k_3)  \otimes  {U}^b (m)\right) \,\Big|\, e_c(n)}_{\mathbb{C}^4}\; e_c(n) , 
\end{equation*}
whence $ \widetilde{\cQ}_2 $ in \eqref{eq:lim_syst_Utilde} acts as a non-local transport between the vectors $ \uU $ and $ \tU $. We have
\begin{lemma} \label{lem:proj_tQ2_e0}
	Let $ U $ be as in Proposition \ref{pr:compact_strong}, and let $ \widetilde{\cQ}_2 $ be defined as in \eqref{eq:limit_cQ2}. Then,
	\begin{equation*}
		\sum_{\substack{n\in\bZ^3 \\ n_h\neq 0}} \ps{\cF\widetilde{\cQ}_2 \pare{\uU, \tU}(n)}{e_0(n)}_{\bC^4}\  e_0(n) = \begin{pmatrix} \underline{u}^h \cdot\nh \uh \\ 0\\ 0 \end{pmatrix} + \begin{pmatrix} \nh \overline{p}_2 \\ 0\\ 0 \end{pmatrix},
	\end{equation*}
	where
	\begin{equation*}
		\overline{p}_2 = \pare{-\Dh}^{-1} \dive_h \pare{\underline{u}^h\cdot \nh \uh}.
	\end{equation*}
\end{lemma}

\begin{proof}

For $ c=\pm $, from \eqref{eq:eigenvectors}, we have $ e^c\perp e_0$, which implies that
\begin{equation*}
	\ps{\cF\widetilde{\cQ}_2 \left( \uU, \tU \right)(n)}{e_0(n)}_{\bC^4} = 2 \sum_{\substack{\pare{0,k_3} + m=n\\m_h, n_h \neq 0 \\ \widetilde{\omega}^{b,0}_{m,n}=0\\b =0, \pm\\j=1,2,3}} \psca{\widehat{\bP}(n) \begin{pmatrix} \cn \\ 0 \end{pmatrix} \cdot \bS \left( {U}^{j} (0,k_3)  \otimes  {U}^b (m)\right) \,\Big|\, e_0 (n) }_{\mathbb{C}^4}. 
\end{equation*}
The condition $ \widetilde{\omega}^{b,0}_{m,n}=0 $ implies that
\begin{equation*}
	\widetilde{\omega}^{b,0}_{m,n} = \omega^b\pare{m}-\omega^0(n) = \pm i \frac{\av{m_h}}{\av{m}} =0,
\end{equation*}
hence $ m_h =0 $. This consideration combined with the convolution constraint $ \pare{0,k_3} + m=n $ implies $ n_h=0 $, which contradicts the definition of the form $ \widetilde{\cQ}_2 $. Then, in the expression of $\cF\widetilde{\cQ}_2 \left( \uU, \tU \right)$, we should only take $ c=0 $ and Lemma \ref{lem:proj_tQ2_e0} follows standard explicit computations.

\end{proof}

Now, setting
\begin{equation*}
	\overline{p}= \overline{p}_1 + \overline{p}_2, 
\end{equation*}
Lemma \ref{lem:proj_tQ1_e0} and \ref{lem:proj_tQ2_e0} imply the following
\begin{cor} \label{cor:proj_cQ1cQ2_e0}
	Let $ U $ be as in Proposition \ref{pr:compact_strong}, and let $ \widetilde{\cQ}_1 $ and $ \widetilde{\cQ}_2 $ be defined as in \eqref{eq:limit_cQ1} and \eqref{eq:limit_cQ2} respectively. Then
	\begin{equation*}
		\sum_{\substack{n\in\bZ^3 \\ n_h\neq 0}} \ps{ \cF \widetilde{\cQ}_1 \pare{\tU, \tU} + \cF \widetilde{\cQ}_2 \pare{U, U}}{e_0(n)}_{\bC^4} \ e_0(n) = \begin{pmatrix} \uh\cdot\nh \uh + \underline{u}^h\cdot\nh \uh \\ 0\\ 0 \end{pmatrix} + \begin{pmatrix} \nh \overline{p} \\ 0\\ 0 \end{pmatrix}.
	\end{equation*}
\end{cor}

\bigskip

\noindent \underline{\textit{Step 2}}: The computation of the projection of $\widetilde{\cA^0_2}\pare{D}\tU$ onto $\text{Span}\set{E_0(n,\cdot)}$ is given in the following lemma, the proof of which can be found in \cite[Lemma 4.4]{Scrobo_Froude_periodic}.
\begin{lemma} \label{lem:proj_A0y_e0}
	Let $ U $ be as in Proposition \eqref{pr:compact_strong} and $ \widetilde{\cA^0_2}\pare{D} $ be defined as in \eqref{eq:limit_At}. Then
	\begin{equation*}
		\overline{\widetilde{\cA^0_2}\pare{D}\tU} = \sum_{\substack{n\in\bZ^3 \\ n_h\neq 0}} \ps{\cF \pare{\widetilde{\cA^0_2}\pare{D}\tU}}{e_0(n)}_{\bC^4} \ e_0(n) = \begin{pmatrix} \nu\Delta\uh\\ 0\\ 0 \end{pmatrix}.
	\end{equation*}
\end{lemma}

\bigskip

\noindent \underline{\textit{Step 3}}: Projecting the system \eqref{eq:lim_syst_Utilde} onto $\text{Span}\set{E_0(n,\cdot)}$ yields the following equation which describes the evolution of $\bU$ 
\begin{equation*}
	\partial_t \bU + \overline{\widetilde{\cQ}_1 \pare{\tU, \tU}} + \overline{\widetilde{\cQ}_2 \pare{U, U}} - \overline{\widetilde{\cA^0_2}\pare{D}\tU} = 0.
\end{equation*}
Then, Corollary \ref{cor:proj_cQ1cQ2_e0} and Lemma \ref{lem:proj_A0y_e0} imply
\begin{prop} \label{prop:limit_bar_local}
	Let $ U $ be as in Proposition \ref{pr:compact_strong} and let 
	\begin{equation*}
		\overline{V}_0 = \sum_{\substack{n\in\bZ^3 \\ n_h\neq 0}} \ps{\cF V_0}{e_0(n)}_{\bC^4} e_0(n) \in \Hs,
	\end{equation*}
	for $ s>5/2 $. Then, the projection $\overline{U}$ of $U$ onto $\text{Span}\set{E_0(n,\cdot)}$ belongs to the energy space 
	\begin{equation*}
		\cC \pare{\left[0, T\right]; H^\sigma \pare{\T^3}}, \ \sigma \in \pare{s-2, s},
	\end{equation*}
	for each $ T\in\left[ 0, T^\star \right[ $, and $ \bU $ solves the Cauchy problem \eqref{eq:Uover_2DstratNS} almost everywhere in $ \T^3\times \left[0, T\right] $.
\end{prop}

\bigskip

\subsubsection{Derivation of the evolution of $ U_{\osc} $}

As for $\bU$, the study of $ U_{\osc} $ also consists in three steps.

\noindent \underline{\textit{Step 1}}: Computation of 
\begin{align*}
 	\pare{\widetilde{\cQ}_1 \pare{\tU, \tU}}_{\osc} &= \sum_{\substack{n\in\bZ^3 \\ n_h\neq 0}} \sum_{c=\pm} \ps{ \cF \widetilde{\cQ}_1 \pare{\tU, \tU}(n)}{e_c(n)}_{\bC^4} e^{i\cn\cdot x} \ e_c(n)\\
 	\pare{\widetilde{\cQ}_2 \pare{\uU, \tU}}_{\osc} &= \sum_{\substack{n\in\bZ^3 \\ n_h\neq 0}} \sum_{c=\pm} \ps{ \cF \widetilde{\cQ}_2 \pare{U, U}(n)}{e_c(n)}_{\bC^4} e^{i\cn\cdot x} \ e_c(n).
\end{align*}
Since $ \tU=\bU+ U_{\osc} $, we can decompose
\begin{equation*}
	\pare{\widetilde{\cQ}_1 \pare{\tU, \tU}}_{\osc} = \pare{\widetilde{\cQ}_1 \pare{\bU, \bU}}_{\osc}+ 2\pare{\widetilde{\cQ}_1 \pare{\bU, U_{\osc}}}_{\osc} + \pare{\widetilde{\cQ}_1 \pare{U_{\osc}, U_{\osc}}}_{\osc}.
\end{equation*}
The first term was already calculated in \cite[Lemma 4.6]{Scrobo_Froude_periodic}, and we have
\begin{lemma} \label{lem:cQ1barbarosc=0}
	The following identity holds true
	\begin{equation*}
		\pare{\widetilde{\cQ}_1 \pare{\bU, \bU}}_{\osc} =0.
	\end{equation*}
\end{lemma}
\noindent To obtain the bilinear term of the equation \eqref{eq:lim_Uosc}, it remains to find the explicit expression of $\pare{\widetilde{\cQ}_2 \pare{\uU, \tU}}_{\osc}$, which is in fact the bilinear term $\cB \left( \uU,  U_{\osc} \right)$ of the equation \eqref{eq:lim_Uosc}.
\begin{lemma}
	We have the following explicit expression
 	\begin{align}
 		\label{eq:tQ2_e+-} \mathcal{F} \ \cB \left( \uU,  U_{\osc} \right) = & \ \mathcal{F} \  \pare{\widetilde{\cQ}_2 \left( \uU,  \tU \right)}_{\osc} \\
 		= & \ \sum_{\substack{ b, c=\pm \\j=1,2,3}} \psca{\widehat{\bP}(n) \begin{pmatrix} \cn \\ 0 \end{pmatrix} \cdot \bS \left( {U}^{j} \left(0, 2n_3 \right)  \otimes  {U}^b \left(n_h,  -n_3\right)\right) \,\Big|\, e_c(n) }_{\mathbb{C}^4}\; e_c(n). \notag 
 	\end{align}
\end{lemma}

\begin{proof}

Since $ \tU=\bU+U_{\osc} $, we have
\begin{equation*}
	\pare{\widetilde{\cQ}_2 \pare{\uU, \tU}}_{\osc} =  \pare{\widetilde{\cQ}_2 \pare{\uU, \bU}}_{\osc} + \pare{\widetilde{\cQ}_2 \pare{\uU, U_{\osc}}}_{\osc}.
\end{equation*}
According to the definition \eqref{eq:limit_cQ2}, we can write
\begin{equation*}
	\pare{\widetilde{\cQ}_2 \pare{\uU, \bU}}_{\osc} = \sum_{\substack{\pare{0,k_3} + m=n\\ m_h, n_h \neq 0\\ \widetilde{\omega}^{0,c}_{m,n}=0\\ c=\pm \\j=1,2,3}} \psca{\widehat{\bP}(n) \begin{pmatrix} \cn \\ 0 \end{pmatrix} \cdot \bS \left( {U}^{j} (0,k_3) \otimes {U}^0 (m)\right) \,\Big|\, e_c(n)}_{\mathbb{C}^4}\; e_c(n).
\end{equation*}
Let us remark that the above formulation differs to the one given in \eqref{eq:limit_cQ2} since the projection onto the oscillating subspace forces the parameter $ c $ to be equal to $ \pm $ only, and the fact that $ \bU $ is the second argument of the bilinear form forces the parameter $ b $ in \eqref{eq:limit_cQ2} to be zero. Hence the bilinear interaction constraint $ \widetilde{\omega}^{0,c}_{m,n}= c \frac{\av{n_h}}{\av{n}}=0 $ combined with the convolution constraint $ \pare{0,k_3} + m=n $ implies that $ n_h=0 $, that contradicts the definition of $ \widetilde{\cQ}_2 $. We deduce that
\begin{equation*}
	\pare{\widetilde{\cQ}_2 \pare{\uU, \bU}}_{\osc} =0.
\end{equation*}

It now remains to prove \eqref{eq:tQ2_e+-}. According to the above argument we can argue that
\begin{align*}
\pare{\widetilde{\cQ}_2 \pare{\uU, \tU}}_{\osc} & =   \pare{\widetilde{\cQ}_2 \pare{\uU, U_{\osc}}}_{\osc}\\
& = \sum_{\substack{\pare{0,k_3} + m=n\\ \widetilde{\omega}^{b,c}_{m,n}=0\\ b, c=\pm \\j=1,2,3}} \left( \left.  \ \widehat{\bP}(n) \begin{pmatrix} \cn \\ 0 \end{pmatrix} \cdot \bS \left( {U}^{j} (0,k_3)  \otimes {U}^b (m)\right) \right| e_c(n) \right)_{\mathbb{C}^4}\; e_c(n).
\end{align*} 
In this case, $n_h = m_h$ and using \eqref{eq:eigenvalues}, the equality $\widetilde{\omega}^{b,c}_{m,n} = \omega^\pm \left( n_h, n_3 \right)- \omega^\pm \left( n_h, m_3 \right)=0 $ becomes
\begin{equation*}
\frac{\left| \cn_h \right|}{\sqrt{\cn_1^2 + \cn_2^2+ \cm_3^2}}
=
\frac{\left| \cn_h \right|}{\sqrt{\cn_1^2 + \cn_2^2+ \cn_3^2}}.
\end{equation*}
The above equality is satisfied if $ m_3=\pm n_3 $. Let us suppose $ m_3 =n_3 $, if this is the case the convolution condition $ k_3  + m_3 = n_3 $ implies that $ k_3=0 $, in this case the term $ \widehat{{U}} \left( 0,0 \right) $ denotes the  average of the element $ {u}^h $, which is identically zero by hypothesis since \eqref{PBSe} propagates the global average which is supposed to be zero since the beginning. Thus, we get $ m_3=-n_3 $ and $ k_3 =2n_3 $ and we recover the expression in \eqref{eq:tQ2_e+-}.

\end{proof}

\bigskip

\noindent \underline{\textit{Step 2}}: Computation of $\pare{\widetilde{\cA^0_2}\pare{D}\tU}_{\osc}$.
\begin{lemma} \label{lem:proj_tA_osc}
	We have
	\begin{equation*}
		\pare{\widetilde{\cA^0_2}\pare{D}\tU}_{\osc} =  \nu \Delta U_{\osc}.
	\end{equation*}
\end{lemma}

\begin{proof}

We will calculate
\begin{equation*}
	\cF \pare{\widetilde{\cA^0_2}\pare{D} U_{\osc}}_{\osc} (n) = \sum_{\substack{a, b =\pm\\ \omega^{a, b}_n=0}}\widehat{U}^a (n) \ps{{\cA_2}(n) e_a(n)}{e_b(n)}_{\bC^4} e_b(n),
\end{equation*}
where the matrix $ \cA_2 $ is defined in \eqref{matrici}.

If $ a=-b $, the condition $ \omega^{a, b}_n=0 $ becomes
\begin{equation*}
	\omega^{a, -a}_n = 2\omega^a(n) = 2 a\ i\frac{\av{n_h}}{\av{n}}=0, 
\end{equation*}
which implies $ n_h=0 $, contradicting the definition of $ \widetilde{\cA^0_2} $. Thus, we deduce that $ a=b $. The expression of the eigenvalues given in \eqref{eq:eigenvalues} implies
\begin{equation*}
\ps{{\cA_2}\pare{D} e_a}{e_a}_{\bC^4} = -\nu \av{n}^2. 
\end{equation*}
Lemma \ref{lem:proj_tA_osc} is then proved.

\end{proof}

\begin{rem}
	We want to emphasize that Lemma \ref{lem:proj_A0y_e0} and \ref{lem:proj_tA_osc} imply the strict (total) parabolicity of the operator $ {\widetilde{\cA^0_2}}\pare{D}$ for vector fields with zero horizontal average. This is remarkable since the operator $ \cA_2 \pare{D} $ appearing in \eqref{PBSe} is \textit{not} strictly parabolic.
\end{rem}

\bigskip

\noindent \underline{\textit{Step 3}}: Projecting the system \eqref{eq:lim_syst_Utilde} onto $\text{Span}\set{E_\pm(n,\cdot)}$ and using Lemma \ref{lem:cQ1barbarosc=0} and \ref{lem:proj_tA_osc}, we deduce that the evolution of $ U_{\osc} $ is given by 
\begin{equation*}
	\partial_t U_{\osc} + \widetilde{\cQ}_1 \pare{U_{\osc}, U_{\osc} + 2 \bU} + \cB\pare{\uU, U_{\osc}}-\nu\Delta U_{\osc}=0,
\end{equation*}
and we hence prove the following result

\begin{prop} \label{prop:limit_osc_local}
	Let $ U_{\osc, 0}\in\Hs, \ s>5/2 $, the projection of $ U $ onto the oscillating subspace $\text{Span}\set{E_\pm(n,\cdot)}$ belongs to the energy space
	\begin{equation*}
		\cC\pare{[0, T]; H^\sigma \pare{\T^3}}, \ \sigma\in\pare{s-2, s}, 
	\end{equation*}  
	for each $ T\in\left[0, T^\star\right[$, and $ U_{\osc} $ solves the Cauchy problem \eqref{eq:lim_Uosc} a.e. in $ \T^3\times [0, T] $.
\end{prop}

\section{Global propagation of smooth data for the limit system}

We already proved in Proposition \ref{prop:limit_eq_Uunderline_local}, \ref{prop:limit_bar_local} and \ref{prop:limit_osc_local}, if the initial data $U_0 \in H^s$, $s > 5/2$, then the decomposition $U=\uU + \bU + U_{\osc}$ holds in $ \cC \pare{[0, T], H^\sigma}$, $s-2<\sigma<s $ and $0\leqslant T<T^\star$, and where $\uU$, $\bU$ and $U_{\osc}$ are respectively solutions of the systems \eqref{eq:Uunder_heat2D}, \eqref{eq:Uover_2DstratNS} and \eqref{eq:lim_Uosc}. The aim of this section is to prove the \textit{global} propagation of the $\Hs$-regularity, $s>5/2$, by the limit system \eqref{eq:limit_system}, more precisely by the systems \eqref{eq:Uunder_heat2D}--\eqref{eq:lim_Uosc}. This propagation can be resumed in the following propositions. We remark that for $\uU$ and $\bU$, we need much less regularity, and the $\Hs$-regularity, $s>5/2$, is especially needed for $U_{\osc}$. 
\begin{prop} \label{pr:global_Hs_Uunder}
	Let $ U_0 \in \Hs, s\geqslant 0 $, then the solution $\uU=\pare{\uu, 0, \underline{U}^4}$ of the equation \eqref{eq:Uunder_heat2D} globally and uniquely exists in time variable
	\begin{equation*}
		\uu \in \cC \pare{\R_+; H^{s}\pare{\T^1_{\v}}} \cap L^2 \pare{\R_+; H^{s+1}\pare{\T^1_{\v}}}, 
	\end{equation*}
	and
	\begin{equation*}
		\underline{U}^4 \in \cC \pare{\R_+; H^s\pare{\T^1_{\v}}}.
	\end{equation*}
\end{prop}

\begin{prop}
	\label{pr:global_Hs_ubarh}
	Let $ U_0 \in\Hs\cap L^\infty \left( \T_v; H^\sigma \left( \T^2_h \right) \right)$, and $ \nh U_0 \in L^\infty \left( \T_v; H^\sigma \left( \T^2_h \right) \right) $ for $s>1/2, \sigma >0$,  then the system \eqref{eq:Uover_2DstratNS} possesses a unique solution in
	\begin{equation*}
		\overline{u}^h \in  \mathcal{C}\left( \mathbb{R}_+;{\cFHs} \right) \cap L^2\left( \mathbb{R}_+; {H}^{s+1}\left( \T^3 \right) \right) .
	\end{equation*}
	Moreover for each $ t>0 $ the following estimate holds true
	\begin{equation} \label{eq:stong_Hs_bound_ubar}
		\left\| \uh \left( t \right)\right\|_{\cFHs}^2 + \nu \int_0^t \left\| \uh\left( \tau \right) \right\|_{H^{s+1}\left( \T^3 \right)}^2d\tau \leqslant  \mathcal{E}_1 \left( U_0 \right),
	\end{equation}
	where
	\begin{equation} \label{eq:E1}
		\mathcal{E}_1 \left( U_0 \right) = C \left\| \uh _0\right\|_{\cFHs}^2 \exp\set{\frac{C  {K}\ \Phi \left( U_0 \right)}{c\nu} \; \left\| \nh \uh_0 \right\|_{L^p_v \left( H^\sigma_h \right)} + \frac{C}{\nu} \left\| \underline{u}^h_0 \right\|^2 _{H^s\left( \T^1_{\v} \right)} }
	\end{equation}
	and
	\begin{equation*} 
		\Phi \left( U_0 \right)= \exp \set{\frac{ C K^2 \left\| \nh\uh_0 \right\|_{L^\infty_v\left( L^2_h \right)}^2 }{c\nu}  \exp \set{ \frac{K}{c\nu} \left( 1+ \left\| \uh_0 \right\|_{L^\infty_v \left( L^2_h \right)}^2 \right) \left\| \nh\uh_0 \right\|_{L^\infty_v \left( L^2_h \right)}^2 } }. 
	\end{equation*}
\end{prop}

\begin{prop} \label{pr:global_Hs_uosc}
	Let $s>5/2 $ and $U_0\in\Hs$. For each $ T>0 $, we have
	\begin{equation*}
		U_{\osc}\in \cC \pare{[0, T]; \Hs}\cap L^2 \pare{[0, T]; H^{s+1}\pare{\T^3}},
	\end{equation*}
	and the following bound holds true for each $ 0\leqslant t\leqslant T $
	\begin{equation*}
		\norm{U_{\osc}\pare{t}}_{\Hs}^2 + \nu \int_0^t \norm{\nabla U_{\osc}\pare{\tau}}_{\Hs}^2d\tau \leqslant \mathcal{E}_{3, \nu, T}\pare{U_0}, 
	\end{equation*}
	where
	\begin{align}
	\label{eq:E3}	\mathcal{E}_{3, \nu, T}\pare{U_0} & = \norm{U_{\osc, 0}}^2_{\Hs} \exp\set{ \frac{1}{\nu}\ \mathcal{E}_1\pare{U_0} + T\ \norm{\uU_0}^2_{L^2\pare{\T^1_{\v}}} + \frac{1}{\nu} \pare{ \cE_{2, U_0}\pare{T} }^2 },\\
	\label{eq:E2}	\cE_{2, U_0}\pare{T} & = C \norm{U_{\osc, 0}}_{\2}^2 \exp \set{\frac{\cE_1\pare{U_0}}{\nu} + T \norm{\uU_0}_{H^s\pare{\T^1_{\v}}}^2 }
	\end{align}
	and $\mathcal{E}_1\pare{U_0}$ is defined as in Proposition \ref{pr:global_Hs_ubarh}.
\end{prop}

\bigskip

\subsection{Proof of Proposition \ref{pr:global_Hs_Uunder}}

The system \eqref{eq:Uover_2DstratNS} is a classical heat equation, the solution of which is well known in the literature. Here, we only remark that classical energy estimates imply
\begin{equation*}
	\norm{\uu\pare{t}}^2_{H^{s}\pare{\T^1_{\v}}} + 2\nu \int_0^t \norm{\partial_3 \uu\pare{\tau}}^2_{H^{s}\pare{\T^1_{\v}}} = \norm{\uu_0}^2_{H^{s}\pare{\T^1_{\v}}}. 
\end{equation*}
Since $ \uu $ has zero vertical average, $ \uu\in L^2 \pare{\R_+ ;H^{s+1}\pare{\T^1_{\v}} } $ as well. \hfill $\square$

\begin{rem}
	We would like to mention that
	\begin{equation} \label{eq:iso_vert_Hs_uunderline}
		\left\| \underline{u}^h \right\|_{H^s \left( \T^1_{\v} \right)} = \left\| \underline{u}^h \right\|_{\cFHs},
	\end{equation}
	hence even if $ \underline{u}^h $ depends on the vertical variable only it still inherits the same isotropic regularity.
\end{rem}

\bigskip

\subsection{Proof of Proposition \ref{pr:global_Hs_ubarh}}

We start by recalling a result proved in \cite[Proposition 5]{Scrobo_Froude_periodic}
\begin{prop} \label{prop:Linfty_integrability_uh}
	Let $\uh$ be a  solution of \eqref{eq:Uover_2DstratNS} with initial data $ \uh_0$ and $\nh \uh_0$  belonging to $L^\infty_v \left( H^\sigma_h \right)$, for some $\sigma\geqslant 1 $.  Then, we have
	\begin{equation*}
		\uh \in L^2 \left( \R_+; {\cFLinfty} \right),
	\end{equation*}
	and in particular
	\begin{equation*}
		\left\| \uh \right\|_{L^2 \left( \R_+; {\cFLinfty} \right)} \leqslant \frac{C{K}}{c\nu} \; \Phi \left( U_0 \right)  \left\| \nh \uh_0 \right\|_{L^p_v \left( H^\sigma_h \right)},
	\end{equation*}
	where $ \Phi \left( U_0 \right) $ is defined as in Proposition \ref{pr:global_Hs_ubarh} and  $c, C,K $ are positive constants.
\end{prop}

\begin{rem}
The reader may notice that \cite[Proposition 5]{Scrobo_Froude_periodic} is applied on a limit system which is slightly different than \eqref{eq:Uover_2DstratNS}, i.e. on the system
\begin{equation}\label{eq:limit_system_other_work}
\left\lbrace
\begin{aligned}
& \partial_t \uh + \uh\cdot\nh \uh -\nu\Delta\uh =-\nabla\bar{p}, \\
& \dive \ \uh =0. 
\end{aligned}
\right. 
\end{equation}
The only difference between \eqref{eq:limit_system_other_work} and \eqref{eq:Uover_2DstratNS} is the presence in \eqref{eq:Uover_2DstratNS} of the term $ \uu\cdot\nh \uh $. Such term though does not pose an obstruction to the application of \cite[Proposition 5]{Scrobo_Froude_periodic} to the limit system \eqref{eq:Uover_2DstratNS}; the proof of such result is in fact based on the fact that the following nonlinear cancellation
\begin{equation*}
\int_{\R^2_h} \pare{\uh\cdot\nh \uh}\cdot \uh \text{d} y_h =0, 
\end{equation*}
holds true for \eqref{eq:limit_system_other_work} (and hence as well for \eqref{eq:Uover_2DstratNS}) since $ \dive \uh =0 $. Indeed though the term  $ \uu\cdot\nh \uh $ enjoys as well a nonlinear cancellation, since
\begin{align*}
\int _{\R_h^2}   \pare{\uu\cdot\nh \uh}\cdot\uh  \text{d} y_h = \frac{1}{2} \ \uu \int _{\R_h^2}\nabla\av{\uh}^2 \text{d} y_h =0, 
\end{align*}
being the vector field periodic. Whence \cite[Proposition 5]{Scrobo_Froude_periodic} can be applied to the limit system \eqref{eq:Uover_2DstratNS}. 
\end{rem}

Next, we need the following estimate
\begin{lemma} 
	Let $ \uh $ be the solution of \eqref{eq:Uover_2DstratNS} and $ \uu $ the solution of \eqref{eq:Uunder_heat2D}, then, for $ s>1/2 $, we have
	\begin{equation} \label{eq:bound_Hs_termine_lineare}
		\left| \psca{\underline{u}^h \cdot \nh \uh \,\big|\, \uh }_{\Hs} \right| \\ \leqslant C  \left(  \left\|  \underline{u}^h \right\|_{H^{s } \left( \T^1_{\v} \right)} +  \left\|  \underline{u}^h \right\|_{H^{s +1} \left( \T^1_{\v} \right)} \right)\left\|  \uh  \right\|_{{\cFHs}} \left\| \nh \uh  \right\|_{{\cFHs}}.
	\end{equation}
\end{lemma}

\begin{proof}

Applying the dyadic cut-off operator $\tq$ to $\underline{u}^h \cdot \nh \uh$, taking the $L^2$-scalar product of the obtain quantity with $\tq \uh$ and applying the Bony decomposition, we get
\begin{equation*}
	\left| \psca{\tq \left( \underline{u}^h \cdot \nh \uh \right) \,\big|\, \tq \uh }_{\cFLtwo} \right| \leqslant B^1_q + B^2_q,
\end{equation*}
where
\begin{align*}
	B^1_q &= \sumf \left| \psca{ \tq \left( S_{q'-1}\underline{u}^h \triangle_{q'}\nh \uh  \right) \,\big|\, \tq \uh  }_{L^2} \right|\\
	B^2_q &= \sumi \left| \psca{ \tq\left( \triangle_{q'}\underline{u}^h S_{q'+2}\nh \uh \right) \,\big|\, \tq \uh  }_{L^2} \right|.
\end{align*}

Applying H\"older inequality and using \eqref{regularity_dyadic} on the term $ B^1_q $, we deduce
\begin{equation*}
	B^1_q \leqslant C b_q 2^{-2qs} \left\|  S_{q'-1}\underline{u}^h \right\|_{L^\infty} \left\| \nh \uh  \right\|_{{\cFHs}}\left\|  \uh  \right\|_{{\cFHs}}.
\end{equation*}
Since $ \underline{u}^h $ only depends on the vertical variable, thanks to the embedding $ H^s \left( \T^1_{\v} \right) \hra L^\infty \left( \T^1_{\v} \right), \ s> 1/2 $, we deduce
\begin{equation*}
 	\left\|  S_{q'-1}\underline{u}^h \right\|_{L^\infty} \leqslant \left\|  \underline{u}^h \right\|_{L^\infty \left( \T^1_{\v} \right)} \leqslant \left\|  \underline{u}^h \right\|_{H^s \left( \T^1_{\v} \right)},
\end{equation*}
and whence,
\begin{equation}
	\label{Bunoq} B^1_q \leqslant C b_q 2^{-2qs} \left\|  \underline{u}^h \right\|_{H^s \left( \T^1_{\v} \right)} \left\| \nh \uh  \right\|_{{\cFHs}}\left\|  \uh  \right\|_{{\cFHs}}.
\end{equation}

Next, we apply H\"older inequality to the term $ B^2_q $ and get 
\begin{equation*}
	B^2_q \leqslant \sumi \left\| \tq \uh \right\|_{{\cFLtwo}} \left\| \triangle_{q'} \underline{u}^h \right\|_{L^2_v \left( L^\infty_h \right)} \left\| \nh \uh \right\|_{L^\infty_v \left( L^2_h \right)}.
\end{equation*}
Bernstein inequality and Estimates \eqref{regularity_dyadic} and \eqref{eq:iso_vert_Hs_uunderline} yield
\begin{equation*}
	\left\| \triangle_{q'} \underline{u}^h \right\|_{L^2_v \left( L^\infty_h \right)} \leqslant C c_q 2^{q'-\left( q'+1 \right)s} \left\| \underline{u}^h \right\|_{H^{s+1}\left( \T^3 \right)} = \ C c_q 2^{-q's} \left\| \underline{u}^h \right\|_{H^{s+1}\left( \T^1_{\v} \right)}.
\end{equation*}
Since $ {\cFHs}\hra H^{0,s}\hra L^\infty_v \left( L^2_v \right), \ s>1/2 $, we have
\begin{equation*}
	\left\| \nh \uh \right\|_{L^\infty_v \left( L^2_h \right)} \leqslant C \left\| \nh \uh \right\|_{{\cFHs}}.
\end{equation*}
Applying once again Estimate \eqref{regularity_dyadic}, we deduce
\begin{equation} \label{Bdueq}
	B^2_q \leqslant  C b_q 2^{-2qs} \left\|  \underline{u}^h \right\|_{H^{s +1} \left( \T^1_{\v} \right)} \left\| \nh \uh  \right\|_{{\cFHs}}\left\|  \uh  \right\|_{{\cFHs}}.
\end{equation}
Now, combining \eqref{Bunoq} and \eqref{Bdueq} finaly implies
\begin{equation*}
	\left| \psca{ \tq \left( \underline{u}^h \cdot \nh \uh \right) \,\big|\, \tq \uh }_{\cFLtwo} \right| \leqslant  C b_q 2^{-2qs} \left( \left\|  \underline{u}^h \right\|_{H^{s } \left( \T^1_{\v} \right)} +  \left\|  \underline{u}^h \right\|_{H^{s +1} \left( \T^1_{\v} \right)} \right)\left\|  \uh  \right\|_{{\cFHs}} \left\| \nh \uh  \right\|_{{\cFHs}}.
\end{equation*}

\end{proof}

\bigskip

\noindent \textit{Proof of Proposition \ref{pr:global_Hs_ubarh}.} We multiply \eqref{eq:Uover_2DstratNS} by $\pare{-\Delta}^{s} \uh$, integrate the obtained quantity over $\mathbb{T}^3$. Using Inequality \eqref{eq:bound_Hs_termine_lineare} and the following inequality
\begin{equation*}
	\av{\ps{\uh\cdot\nh \uh}{\uh}_{\Hs}}\leqslant C \norm{\uh}_{L^\infty}\norm{ \uh}_{\Hs}\norm{\nabla \uh}_{\Hs}, 
\end{equation*}
we deduce that
\begin{multline} \label{ultima?}
	\frac{1}{2}\frac{d}{dt}\left\| \uh \right\|_{\cFHs}^2 + \nu \left\| \uh \right\|_{H^{s+1}\left( \R^3 \right)}^2\\
	\leqslant C \left( \left\| \uh \right\|_{{\cFLinfty}} +\left\|\underline{u}^h \right\|_{H^{s } \left( \T^1_{\v} \right)} + \left\|  \underline{u}^h \right\|_{H^{s +1} \left( \T^1_{\v} \right)}  \right) \left\| \uh \right\|_{\cFHs} \left\| \uh \right\|_{H^{s+1}\left( \R^3 \right)}.
\end{multline}
Then, Young inequality implies
\begin{multline*}
	\left\| \uh \right\|_{{\cFLinfty}} \left\| \uh \right\|_{\cFHs} \left\| \uh \right\|_{H^{s+1}\left( \R^3 \right)} \leqslant \frac{\nu}{2}\left\| \uh \right\|_{H^{s+1}\left( \R^3 \right)}^2\\ 
	+ C \left( \left\| \uh \right\|_{{\cFLinfty}}^2 +\left\|\underline{u}^h \right\|_{H^{s } \left( \T^1_{\v} \right)}^2 + \left\|  \underline{u}^h \right\|_{H^{s +1} \left( \T^1_{\v} \right)}^2  \right) \left\| \uh \right\|_{\cFHs}^2,
\end{multline*}
which, together with \eqref{ultima?} and Gronwall lemma, leads to
\begin{multline*}
	\left\| \uh \left( t \right)\right\|_{\cFHs}^2 + \nu \int_0^t \left\| \uh\left( \tau \right) \right\|_{H^{s+1}\left( \T^3 \right)}^2d\tau \\
	\leqslant C \left\| \uh _0\right\|_{\cFHs}^2 \exp\left\{\int_0^t \left\| \uh \left( \tau \right) \right\|_{{\cFLinfty}}^2  +\left\|\underline{u}^h\left( \tau \right) \right\|_{H^{s } \left( \T^1_{\v} \right)}^2 + \left\|  \underline{u}^h\left( \tau \right) \right\|_{H^{s +1} \left( \T^1_{\v} \right)}^2 d\tau\right\}.
\end{multline*}
Using Proposition \ref{prop:Linfty_integrability_uh} and Proposition \ref{pr:global_Hs_Uunder}, we finaly obtain
\begin{multline*}
	\left\| \uh \left( t \right)\right\|_{\cFHs}^2 + \nu \int_0^t \left\| \uh\left( \tau \right) \right\|_{H^{s+1}\left( \T^3 \right)}^2d\tau\\
	\leqslant C \left\| \uh _0\right\|_{\cFHs}^2 \exp\set{\frac{C  {K}}{c\nu} \; \Phi \left( U_0 \right) \left\| \nh \uh_0 \right\|_{L^p_v \left( H^\sigma_h \right)} + \frac{C}{\nu} \left\| \underline{u}^h_0 \right\| _{H^s\left( \T^1_{\v} \right)} },
\end{multline*}
where $ \Phi $ is defined as in Proposition \ref{pr:global_Hs_ubarh}. \hfill $ \Box $

\subsection{Proof of Proposition \ref{pr:global_Hs_uosc}}

We first remark that, if $ \tU $ and $ \uU $ are smooth enough, the system \eqref{eq:lim_Uosc} admits global weak solutions \textit{\`a la Leray} in the same fashion as for the incompressible \NS\ equations (see \cite{monographrotating} for instance).
\begin{lemma} \label{lem:Ler_sol_Uosc}
	Let $s > 1/2$, $\bU \in L^2\pare{\R_+; H^{s+1}\pare{\T^3}}$ and $\uU \in L^\infty\pare{\R_+; H^{s}\pare{\T^1_{\v}}}$. Then, for any initial data $U_{\osc, 0}\in\2$, there exists a global weak solution of the system \eqref{eq:lim_Uosc} such that
	\begin{equation*}
		U_{\osc}\in \cC_{\loc}\pare{\R_+; \2}\cap L^2_{\loc}\pare{\R_+; H^1\pare{\T^3}}.
	\end{equation*}
	Moreover for any $t^\star \in \R_+$ and for any $ 0\leqslant t \leqslant t^\star < \infty $, the following estimate holds true
	\begin{equation} \label{eq:L2_bound_Uosc}
		\norm{U_{\osc}\pare{t}}_{\2}^2 + \nu \int_0^t \norm{ \nabla U_{\osc}\pare{\tau}}_{\2}^2 d\tau \leqslant \cE_{2, U_0} \pare{t^\star}.
	\end{equation}
	where 
	\begin{equation*}
		\cE_{2, U_0} \pare{t^\star} = C \norm{U_{\osc, 0}}_{\2}^2 \exp \set{\frac{\cE_1\pare{U_0}}{\nu} + t^\star \norm{\uU_0}_{H^s\pare{\T^1_{\v}}}^2 }.
	\end{equation*}
\end{lemma}

\begin{proof}

We define the frequency cut-off operator
\begin{equation*}
	J_n W = \sum_{\av{k}\leqslant n} \widehat{W}(n) e^{i \cn\cdot x}, 
\end{equation*}
and consider the approximate system
\begin{equation} \label{eq:approx_osc_system}
	\left\lbrace
	\begin{aligned}
		&\partial_t U_{\osc, n} +J_n \widetilde{\mathcal{Q}}_1\left(U_{\osc, n} +2 \overline{U}, U_{\osc, n} \right) +J_n \mathcal{B} \left( \uU, U_{\osc, n} \right) -  \nu \Delta U_{\osc, n} =0,\\
		& \dive U_{\osc, n}=0, \\
		&\left. U_{\osc, n} \right|_{t=0}=J_n U_{{\osc}, 0}.
	\end{aligned}
	\right.
\end{equation}
The Cauchy-Lipschitz theorem implies the existence of a local solution for \eqref{eq:approx_osc_system} in the space
\begin{equation*}
	U_{\osc, n} \in \cC\pare{\left[0, T_n\right]; L^2_n},
\end{equation*}
where
\begin{equation*}
	L^2_n = \set{f \in L^2(\T^3), \text{supp } \widehat{f} \subset \cB(0,n)}.
\end{equation*}
Since $ U_{\osc, n} $ is of divergence-free we deduce that
\begin{equation*}
\ps{J_n \widetilde{\cQ}_1 \pare{U_{\osc, n}, U_{\osc, n}}}{U_{\osc, n}}_{\2}=0.
\end{equation*}
Moreover, the following inequalities hold true thanks to the embedding $ H^s\hra L^\infty, \ s>\frac{d}{2} $;
\begin{align*}
	\ps{\widetilde{\cQ}_1 \pare{\bU, U_{\osc,n}}}{U_{\osc, n}}& \leqslant C \norm{\nabla \bU}_{\Hs} \norm{U_{\osc, n}}_{\2}\norm{\nabla U_{\osc, n}}_{\2},\\
	\ps{\cB \pare{\uU, U_{\osc,n}}}{U_{\osc, n}}& \leqslant C \norm{\uU}_{H^s\pare{\T^1_{\v}}} \norm{U_{\osc, n}}_{\2}\norm{\nabla U_{\osc, n}}_{\2},
\end{align*} 
which yield, for any $ t^\star \in \R_+ $ and $ t \in [0, t^\star[ $,
\begin{align*}
	&\norm{U_{\osc, n}\pare{t}}_{\2}^2 + \nu \int_0^t \norm{ \nabla U_{\osc, n}\pare{\tau}}_{\2}^2 d\tau\\
	&\hspace{3cm} \leqslant C \norm{U_{\osc, 0}}_{\2}^2 \exp \set{ \int_0^t \norm{\nabla \bU\pare{\tau}}_{\Hs}^2 + \norm{\uU\pare{\tau}}_{H^s\pare{\T^1_{\v}}}^2 d\tau }, \\
	&\hspace{3cm} \leqslant C \norm{U_{\osc, 0}}_{\2}^2 \exp \set{ \frac{\cE_1\pare{U_0}}{\nu} + t^\star \norm{\uU_0}_{H^s\pare{\T^1_{\v}}}^2 },
\end{align*}
where $ \mathcal{E}_1 $ is defined in \eqref{eq:E1}. Hence, by a continuation argument, we deduce that $ T_n=\infty $ and for each $ T>0 $, the sequence $ \pare{U_{\osc, n}}_n $ is uniformly bounded in the space 
\begin{equation*}
\cC\pare{[0, T]; \2}\cap L^2 \pare{[0, T]; H^1\pare{\T^3}}.
\end{equation*}
Standard product rules in Sobolev spaces show that the sequence $ \pare{\partial_t U_{\osc, n}}_n $ is uniformly bounded in the space $ L^2 \pare{[0, T]; H^{-N}} $ for $ N \in \mathbb{N} $ large enough. Finaly, applying Aubin-Lions lemma (see \cite{Aubin63}), we deduce that the sequence $ \pare{U_{\osc, n}}_n $ is compact in $ L^2 \pare{[0, T]; L^2} $, and each limit point of $ \pare{U_{\osc, n}}_n $ weakly solves \eqref{eq:lim_Uosc}.

\end{proof}

\begin{rem}
	We point out that the above construction of global weak solutions is possible thanks to the presence of the uniformly parabolic smoothing effect on the limit system \eqref{eq:lim_Uosc}, hence the importance of the propagation of parabolicity mentioned in Remark \ref{rem:propagation_parabolicity}. 
\end{rem}

Next, we study the ``purely bilinear'' interactions of highly oscillating perturbations in \eqref{eq:lim_syst_Utilde} given by the term
\begin{equation*}
	\pare{\widetilde{\cQ}_1 \pare{U_{\osc}, U_{\osc}}}_{\osc}. 
\end{equation*}
Bilinear interactions of the above form, in general, prevent us from obtaining global-in-time energy subcritical and critical estimates. However, as pointed out in Remark \ref{rem:smoothness_bilinear_Fourier}, we can actually prove that the bilinear interaction $ \widetilde{\cQ}_1 \pare{U_{\osc}, U_{\osc}} $ is in fact smoother than the vector $ U_{\osc}\cdot\nabla U_{\osc} $. To do so, we introduce the following \textit{resonant set}. 
\begin{definition} \label{resonance set}

	\begin{enumerate}
		\item The resonant set $\mathcal{K}^\star$ is the set of frequencies such that
		\begin{align*}
			\mathcal{K}^\star &= \left\lbrace \left( k,m,n \right)\in \mathbb{Z}^9, k_h, m_h, n_h\neq 0 \left|\hspace{3mm} \omega^a(k)+ \omega^b(m)=\omega^c(n) , \ k+m=n, \hspace{3mm} \left( a,b,c\right) \in \left\lbrace -,+ \right\rbrace  \right.\right\rbrace,\\
			&= \left\lbrace \left( k,n \right)\in \mathbb{Z}^6, k_h,  n_h\neq 0 \left|\hspace{3mm} \omega^a(k)+ \omega^b(n-k)=\omega^c(n), \hspace{3mm} \left( a,b,c\right) \in \left\lbrace -,+ \right\rbrace  \right.\right\rbrace, 
		\end{align*}
		where $\omega^j, \ j=\pm$ are the eigenvalues given in \eqref{eq:eigenvalues}. 

		\item The \textit{resonant set of the frequency}  $n: n_h\neq 0$, is defined as
		\begin{equation*}
			\mathcal{K}^\star_n = \left\lbrace \left( k,m \right)\in \mathbb{Z}^6 \left|\hspace{3mm} \omega^a(k)+ \omega^b(m)=\omega^c(n) \text{ with } k+m=n, \hspace{3mm} \left( a,b,c\right) \in \left\lbrace -,+ \right\rbrace 			 \right.\right\rbrace.
		\end{equation*}
	\end{enumerate}

\end{definition}

\noindent The resonant set is introduced in order to express the term $ \pare{\widetilde{\cQ}_1 \pare{U_{\osc}, U_{\osc}}}_{\osc} $ in a more concise way. Indeed, considering the explicit definition of the bilinear form $ \widetilde{\cQ}_1 $ given in \eqref{eq:limit_cQ1} we can immediately deduce that
\begin{equation*}
	\pare{\widetilde{\cQ}_1 \pare{U_{\osc}, U_{\osc}}}_{\osc} = \cF^{-1}\pare{1_{\mathcal{K}^\star} \cF \pare{U_{\osc}\cdot \nabla U_{\osc}}}. 
\end{equation*}
In other words, the resonant set $ \mathcal{K}^\star $ is the set of frequencies on which the bilinear interaction $ \pare{\widetilde{\cQ}_1 \pare{U_{\osc}, U_{\osc}}}_{\osc} $ is localized.

We now define the following Fourier multiplier of order zero
\begin{equation*}
	\chi_{\mathcal{K}^\star}\pare{D} \pare{ a \ b} = \cF^{-1}\pare{1_{\mathcal{K}^\star} \cF \pare{a \ b}}.
\end{equation*}
We can hence rewrite
\begin{equation*} 
	\pare{\widetilde{\cQ}_1 \pare{U_{\osc}, U_{\osc}}}_{\osc} = \dive \left[ \chi_{\mathcal{K}^\star}\pare{D} \pare{U_{\osc}\otimes U_{\osc}}\right].
\end{equation*}
We state the following technical lemma which is a simple variation of \cite[Lemma 6.6, p.150]{monographrotating},  \cite[Lemma 6.4, p.222]{paicu_rotating_fluids} or \cite[Lemma 8.4]{Scrobo_primitive_horizontal_viscosity_periodic}. The proof is based on the fact that, for fixed $ \left( k_h, n \right) $, the fiber 
\begin{equation*} 
	\mathcal{J}\left(  k_h, n \right)= \left\{ k_3 \in \bZ \,, \left( k,n \right)\in \mathcal{K}^\star \right\}
\end{equation*} 
is a finite set.
\begin{lemma} \label{lem:product_rule_osc}
	Let $a,b \in H^{1/2}\left(\mathbb{T}^3\right)$ and $c\in \2$ be vector fields of zero horizontal average on $\mathbb{T}^2_{\textnormal{h}}$.  Then there exists a constant $C$ which only depends on $a_1/a_2$ such that
	\begin{equation} \label{eq:product_rule_osc}
		\left| \sum_{(k,n)\in \mathcal{K}^\star } \widehat{a}(k) \widehat{b}\left( {n-k} \right) \widehat{c}(n) \right| \leqslant \frac{C}{a_3} \left\| a\right\|_{H^{1/2}\left(\mathbb{T}^3\right)} \left\|b \right\|_{H^{1/2}\left(\mathbb{T}^3\right)} \left\|c\right\|_\2. 
	\end{equation}
\end{lemma}

\begin{proof}

We first prove Lemma \ref{lem:product_rule_osc} when $\T^3=\left[0,2\pi\right)^3$. We write
\begin{align}
	\label{res ineq 1} I_{\mathcal{K}^\star}= \left| \sum_{(k,n)\in \mathcal{K}^\star } \widehat{a}_k \widehat{b}_{n-k} \widehat{c}_n \right|\leqslant & \sum_{\left( k_h,n\right)\in \mathbb{Z}^2\times \mathbb{Z}^3} \sum_{\left\lbrace k_3:(k,n)\in \mathcal{K}^\star\right\rbrace } \left| \widehat{a}_k \widehat{b}_{n-k} \widehat{c}_n \right|,\\
	\leqslant & \sum_{\left( k_h,n\right)\in \mathbb{Z}^2\times \mathbb{Z}^3} \left| \widehat{c}_n\right| \sum_{\left\lbrace k_3:(k,n)\in \mathcal{K}^\star\right\rbrace } 
	\left|\widehat{a}_k\right|\left| \widehat{b}_{n-k}\right|.\notag
\end{align}

By Cauchy-Schwarz inequality, we have
\begin{equation*}
	\sum_{\left\lbrace k_3:(k,n)\in \mathcal{K}^\star\right\rbrace } 
	\left|\widehat{a}_k\right|\left| \widehat{b}_{n-k}\right|\leqslant
	\left( \sum_{\left\lbrace k_3:(k,n)\in \mathcal{K}^\star\right\rbrace } 
	\left|\widehat{a}_k\right|^2\left| \widehat{b}_{n-k}\right|^2\right)^{1/2} 
	\left( \sum_{\left\lbrace k_3:(k,n)\in \mathcal{K}^\star\right\rbrace }  1\right)^{1/2}.
\end{equation*}
Now, fixing $\left( k_h,n\right)\in \mathbb{Z}^2\times \mathbb{Z}^3$ there exists only a finite number of resonant modes $k_3$, more precisely, 
\begin{equation}
	\label{eq:nok3} \# \left(\left\lbrace k_3:(k,n)\in \mathcal{K}^\star\right\rbrace\right) \leqslant 8.
\end{equation}
Indeed, we can write explicitly the resonant condition $ \omega^{+,+,+}_{k, n-k,n}=0 $ (the same procedure holds for the generic case $ \omega^{a,b,c}_{k, n-k,n}=0, a,b,c\neq 0 $) as follows
\begin{equation*}
	\left(\frac{\left|k_h\right|^2}{\left|k_3\right|^2+\left|k_h\right|^2}\right)^{1/2}
	+\left(\frac{ \left| n_h-k_h\right|^2}{\left| n_3-k_3\right|^2+\left|n_h-k_h\right|^2}\right)^{1/2}
	= \left(\frac{\left\vert n_h\right|^2}{\left|n_3\right|^2+\left| n_h\right|^2}\right)^{1/2}.
\end{equation*}
After some algebraic calculations, the above equation of $k_3$ ($k_h$ and $n$ being fixed) becomes an polynomial equation of the form
\begin{equation*}
	R\left( k_3 \right)=0,
\end{equation*}
where $R$ is a real polynomial of degree eight, hence \eqref{eq:nok3} follows the fundamental theorem of algebra. Thus,
\begin{equation*}
	\sum_{\left\lbrace k_3:(k,n)\in \mathcal{K}^\star\right\rbrace } 	\left|\widehat{a}_k\right|\left| \widehat{b}_{n-k}\right|\leqslant 	\sqrt{8}\left( \sum_{\left\lbrace k_3:(k,n)\in \mathcal{K}^\star\right\rbrace } 	\left|\widehat{a}_k\right|^2\left| \widehat{b}_{n-k}\right|^2\right)^{1/2},
\end{equation*}
which, combined with Inequality \eqref{res ineq 1}, gives
\begin{equation*}
	I_{\mathcal{K}^\star} \leqslant \sqrt{8}  \sum_{k_h,n_h} \sum_{n_3 } \left| \widehat{c}_n\right| \left( \sum_{k_3 } \left|\widehat{a}_k\right|^2\left| \widehat{b}_{n-k}\right|^2\right)^{1/2}.
\end{equation*}
Moreover
\begin{equation*}
	\sum_{n_3 } \left| \widehat{c}_n\right| \left( \sum_{k_3 } 	\left|\widehat{a}_k\right|^2\left| \widehat{b}_{n-k}\right|^2\right)^{1/2} \leqslant \left( \sum_{n_3 } \left| \widehat{c}_n\right|^2 \right)^{1/2} \left(\sum_{n_3, k_3 } \left|\widehat{a}_k\right|^2\left| \widehat{b}_{n-k}\right|^2\right)^{1/2},
\end{equation*}
and hence
\begin{equation} \label{res ineq 2}
	I_{\mathcal{K}^\star} \leqslant \sqrt{8} \sum_{\left( k_h,n\right)\in \mathbb{Z}^2\times \mathbb{Z}^3} \left( \sum_{n_3 } \left| \widehat{c}_n\right|^2 \right)^{1/2} \left( \sum_{p_3}\left| \widehat{b}_{n_h-k_h,p_3}\right|^2\right)^{1/2} \left( \sum_{k_3} \left| \widehat{a}_k\right|^2\right)^{1/2}.
\end{equation}

Let us denote at this point
\begin{equation*}
	\widetilde{a}_{n_h} = \left( \sum_{n_3} \left| \widehat{a}_n\right|^2\right)^{1/2}, \hspace{2cm} 
	\widetilde{b}_{n_h} = \left( \sum_{n_3} \left| \widehat{b}_n\right|^2\right)^{1/2}, \hspace{2cm} 
	\widetilde{c}_{n_h} = \left( \sum_{n_3} \left| \widehat{c}_n\right|^2\right)^{1/2},
\end{equation*}
and the following distributions
\begin{equation*}
	\widetilde{a}\left( x_h\right) = \mathcal{F}_h^{-1}\left(\widetilde{a}_{n_h} \right) \hspace{2cm}
	\widetilde{b}\left( x_h\right) = \mathcal{F}_h^{-1}\left(\widetilde{b}_{n_h} \right) \hspace{2cm}
	\widetilde{c}\left( x_h\right) = \mathcal{F}_h^{-1}\left(\widetilde{c}_{n_h} \right). 
\end{equation*}
The inequality \eqref{res ineq 2} can be read, applying Plancherel theorem and the product rules for Sobolev spaces, as
\begin{align*}
	I_{\mathcal{K}^\star}\leqslant \psca{\widetilde{a}\widetilde{b} \,\big|\, \widetilde{c}}_{L^2\left(\mathbb{T}^2_{\textnormal{h}}\right)} \leqslant & \left\| \widetilde{a}\widetilde{b} \right\|_{L^2\left(\mathbb{T}^2_{\textnormal{h}}\right)}\left\|\widetilde{c}\right\|_{L^2\left(\mathbb{T}^2_{\textnormal{h}}\right)}\\
	\leqslant & \left\| \widetilde{a}\right\|_{H^{1/2}\left(\mathbb{T}^2_{\textnormal{h}}\right)}\left\|\widetilde{b} \right\|_{H^{1/2}\left(\mathbb{T}^2_{\textnormal{h}}\right)}\left\|\widetilde{c}\right\|_{L^2\left(\mathbb{T}^2_{\textnormal{h}}\right)}\\
	=& \left\| a\right\|_{H^{1/2,0}\left(\mathbb{T}^3\right)} \left\|b \right\|_{H^{1/2,0}\left(\mathbb{T}^3\right)} \left\|c\right\|_{\2}, \\
	\leqslant & \left\| a\right\|_{H^{1/2}\left(\mathbb{T}^3\right)} \left\|b \right\|_{H^{1/2}\left(\mathbb{T}^3\right)} \left\|c\right\|_{\2}.
\end{align*}

Finaly, to lift this argument to a generic torus $\prod_{i=1}^3 \left[ 0, 2\pi a_i\right)$, it suffices to use the transform 
$$
\widetilde{v}\left( x_1, x_2, x_3\right) = v\left( a_1x_1, a_2 x_2, a_3 x_3\right),
$$
and the identity
$$
\left\| \widetilde{v}\right\|_{L^2\left( \left[ 0,2\pi\right)^3\right)}= \left( a_1a_2a_3\right)^{-1/2} \left\| v\right\|_{L^2\left( \prod_{i=1}^3\left[ 0, 2\pi a_i\right)\right)}.
$$

\end{proof}

\begin{rem}
	Lemma \eqref{lem:product_rule_osc} can be applied on $ U_{\osc} $, by taking $ a=b=c=U_{\osc} $,  since the projection on the oscillating subspace defined in \eqref{eq:DecompW} has zero horizontal average. 
\end{rem}

Now, we can prove the energy bound required on the problematic trilinear term
\begin{lemma}
	Let $ s>0 $, then
	\begin{equation}\label{eq:estimate_trilinear_term_osc}
		\ps{\pare{\widetilde{\cQ}_1 \pare{U_{\osc}, U_{\osc}}}_{\osc}}{U_{\osc}}_{\Hs} \leqslant C \norm{U_{\osc}}_{\2}^{1/2} \norm{\nabla U_{\osc}}_{\2}^{1/2}\norm{U_{\osc}}_{\Hs}^{1/2} \norm{\nabla U_{\osc}}_{\Hs}^{3/2}. 
	\end{equation}
\end{lemma}

\begin{proof}

We remark that
\begin{align*}
	\ps{\pare{\widetilde{\cQ}_1 \pare{U_{\osc}, U_{\osc}}}_{\osc}}{U_{\osc}}_{\Hs} & = \ps{\widetilde{\cQ}_1 \pare{U_{\osc}, U_{\osc}}}{U_{\osc}}_{\Hs}, \\
	& = - \ps{\chi_{\mathcal{K}^\star}\pare{D} \pare{U_{\osc}\otimes U_{\osc}}}{\nabla U_{\osc}}_{\Hs}, \\
 	& = - \psc{ \pare{-\Delta}^{s/2} \pare{U_{\osc}\otimes U_{\osc}}}{ \pare{-\Delta}^{s/2}\nabla U_{\osc}}_{\chi_{\mathcal{K}^\star}}, 
\end{align*}
where
\begin{equation*}
	\psc{a\ b}{c}_{\chi_{\mathcal{K}^\star}} = \sum_{(k,n)\in \mathcal{K}^\star } \widehat{a}(k) \widehat{b}\left( {n-k} \right) \widehat{c}(n).
\end{equation*}
By a dyadic decomposition, we also have
\begin{equation*}
	\av{\psc{ \pare{-\Delta}^{s/2} \pare{U_{\osc}\otimes U_{\osc}}}{ \pare{-\Delta}^{s/2}\nabla U_{\osc}}_{\chi_{\mathcal{K}^\star}}}\sim \sum_q 2 ^{2qs}  \ \av{ \psc{ \tq \pare{U_{\osc}\otimes U_{\osc}}}{ \tq\nabla U_{\osc}}_{\chi_{\mathcal{K}^\star}}}. 
\end{equation*}
For each dyadic bloc in the above estimate, using a Bony decomposition, we have
\begin{equation*}
	I_q = \av{ \psc{ \tq \pare{U_{\osc}\otimes U_{\osc}}}{ \tq\nabla U_{\osc}}_{\chi_{\mathcal{K}^\star}}} \leqslant I_q^1 + I_q^2,
\end{equation*}
where
\begin{align*}
	I_q^1 &= \sumf \av{ \psc{ \tq \pare{S_{q'}U_{\osc}\otimes \Tq U_{\osc}}}{ \tq\nabla U_{\osc}}_{\chi_{\mathcal{K}^\star}}}\\
	I_q^2 &= \sumi \av{ \psc{ \tq \pare{\Tq U_{\osc}\otimes S_{q'+2} U_{\osc}}}{ \tq\nabla U_{\osc}}_{\chi_{\mathcal{K}^\star}}}.
\end{align*}
Combining \eqref{eq:product_rule_osc} with some classical computations with the dyadic blocs finaly leads to, for any $k=1,2$,
\begin{equation*}
	I_q^k \leqslant C \ 2^{-2qs}b_q \ \norm{U_{\osc}}_{\2}^{1/2} \norm{\nabla U_{\osc}}_{\2}^{1/2}\norm{U_{\osc}}_{\Hs}^{1/2} \norm{\nabla U_{\osc}}_{\Hs}^{3/2}, 
\end{equation*}
where the sequence $ \pare{b_q}_q\in \ell^2 $ depends on $ U_{\osc} $, concluding the proof.

\end{proof}

\begin{lemma}
	\label{lem:613}
	Let $ s >1/2 $,  then
	\begin{align}
		\ps{\widetilde{\cQ}_1\pare{\bU, U_{\osc}}}{U_{\osc}}_{\Hs}&  \leqslant C \norm{\nabla\bU}_{\Hs}\norm{\nabla U_{\osc}}_{\Hs}\norm{ U_{\osc}}_{\Hs},\label{eq:estimate_linear_term_osc1} \\
		\ps{\cB\pare{\uU, U_{\osc}}}{U_{\osc}}_{\Hs} & \leqslant C \norm{\uU}_{L^2\pare{\T^1_{\v}}} \norm{\nabla U_{\osc}}_{\Hs}\norm{ U_{\osc}}_{\Hs}.\label{eq:estimate_linear_term_osc2}
	\end{align}
\end{lemma}

\begin{proof}

The proof of Lemma \ref{lem:613} relies on direct estimates performed on both bilinear terms. For the first one, we have
\begin{align*}
	\av{\ps{\widetilde{\cQ}_1\pare{\bU, U_{\osc}}}{U_{\osc}}_{\Hs}}& \leqslant \av{\ps{\dive \pare{\bU\otimes U_{\osc}}}{U\osc}_{\Hs}}, \\
	& \leqslant \norm{\bU\otimes U_{\osc}}_{H^{s+1}\pare{\T^3}} \norm{U_{\osc}}_{\Hs},\\
 	& \leqslant C \norm{\nabla\bU}_{\Hs}\norm{\nabla U_{\osc}}_{\Hs}\norm{ U_{\osc}}_{\Hs}, 
\end{align*}
where in the last inequality, we used the fact that $ H^{s+1}\pare{\T^3}, s>1/2 $ is a Banach algebra.

For the second one we use the explicit definition of the limit bilinear form $ \mathcal{B} $ given in \eqref{eq:tQ2_e+-} in order to deduce the identity
\begin{align*}
	\ps{\cB\pare{\uU, U_{\osc}}}{U_{\osc}}_{\Hs} & = \ps{\pare{-\Delta}^{s/2}\cB\pare{\uU, U_{\osc}}}{\pare{-\Delta}^{s/2} U_{\osc}}_{\2}, \\
	& = \ps{\cB\pare{\uU, \pare{-\Delta}^{s/2} U_{\osc}}}{\pare{-\Delta}^{s/2} U_{\osc}}_{\2}, 
\end{align*}
which implies inequality \eqref{eq:estimate_linear_term_osc2}.

\end{proof}

\noindent \textit{Proof of Proposition \ref{pr:global_Hs_uosc}.} We have now all the ingredients to prove Proposition \ref{pr:global_Hs_uosc}. Performing rather standard $\Hs$-energy estimates on the equation \eqref{eq:lim_Uosc} with the energy bounds \eqref{eq:estimate_trilinear_term_osc}, \eqref{eq:estimate_linear_term_osc1} and \eqref{eq:estimate_linear_term_osc2}, we obtain
\begin{multline*}
	\frac{1}{2}\frac{d}{dt} \norm{U_{\osc}\pare{t}}_{\Hs}^2 + \nu \int_0^t \norm{\nabla U_{\osc} \pare{\tau}}_{\Hs}^2 d\tau \\
	\leqslant C \pare{ \norm{\nabla\bU}_{\Hs} + \norm{\uU}_{L^2\pare{\T^1_{\v}}} }\norm{\nabla U_{\osc}}_{\Hs}\norm{ U_{\osc}}_{\Hs} \\
	+ \norm{U_{\osc}}_{\2}^{1/2} \norm{\nabla U_{\osc}}_{\2}^{1/2}\norm{U_{\osc}}_{\Hs}^{1/2} \norm{\nabla U_{\osc}}_{\Hs}^{3/2}
\end{multline*}
Then, Young inequality and Gronwall lemma imply
\begin{multline*}
	\norm{U_{\osc}\pare{\tau}}_{\Hs}^2 + \nu \int_0^t \norm{\nabla U_{\osc}\pare{\tau}}_{\Hs}^2d\tau\\
	\leqslant \norm{U_{\osc, 0}}^2_{\Hs} \exp\left\{ \int_0^t \norm{\nabla\uh\pare{\tau}}_{\Hs}^2d\tau + \int_0^t\norm{\uU\pare{\tau}}_{L^2\pare{\T^1_{\v}}}^2d\tau\right.\\
	+ \left.\int_0^t \norm{U_{\osc}\pare{\tau}}_{\2}^2\norm{\nabla U_{\osc}\pare{\tau}}_{\2}^2 d\tau \right\}.
\end{multline*}
Thus, using Estimates \eqref{eq:stong_Hs_bound_ubar}, \eqref{eq:L2_bound_Uosc} and the result in Proposition \ref{pr:global_Hs_Uunder}, we deduce that, for each $ T> 0 $, the following bound holds true
\begin{multline*}
	\norm{U_{\osc}\pare{\tau}}_{\Hs}^2 + \nu \int_0^t \norm{\nabla U_{\osc}\pare{\tau}}_{\Hs}^2d\tau \\
	\begin{aligned}
		\leqslant & \  \norm{U_{\osc, 0}}^2_{\Hs} \exp\set{ \frac{1}{\nu}\ \mathcal{E}_1\pare{U_0} + T\ \norm{\uU_0}^2_{L^2\pare{\T^1_{\v}}} + \frac{1}{\nu} \pare{ \cE_{2, U_0}\pare{T} }^2 }, \\
		\leqslant &  \ C_\nu \exp\set{C_\nu \exp\set{C_\nu T}},
	\end{aligned}
\end{multline*}
where $ \mathcal{E}_1 $ and $ \mathcal{E}_2 $ are respectively defined in \eqref{eq:E1} and \eqref{eq:E2}. 
\hfill$ \Box $

\section{Convergence as $ \varepsilon\to 0 $ and proof of the main result}

As in the work \cite{Gallagher_incompressible_limit}, the lack of a complete parabolic smoothing effect on the system \eqref{PBSe} will prevent us to obtain a uniform global-in-time control for $ U^\varepsilon $. Nonetheless we will be able to prove that, for each $ T>0 $ arbitrary and $ \varepsilon > 0 $, the solutions of \eqref{eq:filt-sys} belong to the space $ \cC_{\loc} \pare{\R_+; H^{s-2}\pare{\T^3}} $ for $ s>9/2 $ and converge in the same topology to the global solution of \eqref{eq:limit_system}. The idea to prove this convergence result is to use the method of Schochet (see \cite{schochet}), which consists in a smart change of variable, which cancels some perturbations that we cannot control. We will use results and terminology introduced by I. Gallagher in \cite{Gallagher_singular_hyperbolic} in the context of quasilinear hyperbolic symmetric systems with skew-symmetric singular perturbation.

Let us recall the following definition \cite[Definition 1.2]{Gallagher_singular_hyperbolic}
\begin{definition} \label{def:oscillating_functions}
	Let $ T, \varepsilon_0> 0, \ p\geqslant 1 $ and $ \sigma > d/2 $. Let $ \overrightarrow{k_q}= \pare{k_1, \ldots , k_q} $ where $ k_i\in\bZ^d $ and let
	\begin{equation*}
		\av{\overrightarrow{k_q}}= \max_{1\leqslant i \leqslant q} \av{k_i}.
	\end{equation*}
	Then a function $ R^\varepsilon_{\osc}\pare{t} $ is said to be \textnormal{$ \pare{p, \sigma} $-- oscillating function} if it can be written as
	\begin{equation*}
		R^\varepsilon_{\osc} = \sum_{q=1}^p R^\varepsilon_{q, \osc}\pare{t},
	\end{equation*}
	where
	\begin{equation*}
		R^\varepsilon_{q, \osc}\pare{t} = \cF^{-1} \pare{\sum_{\overrightarrow{k_q} \in K^n_q} e^{i\frac{t}{\varepsilon}\beta_q \pare{n, \overrightarrow{k_q}} } r_0\pare{n , \overrightarrow{k_q}} f^\varepsilon_1 \pare{t, k_1}\ldots f^\varepsilon_q \pare{t, k_q} },  
	\end{equation*}
	with
	\begin{equation*}
		K^n_q = \set{ \overrightarrow{k_q}\in \bZ^{dq} \left\vert \ \sum_{i=1}^q k_i =n \text{ and } \beta_q\pare{n, \overrightarrow{k_q}}\neq 0 \right. },
	\end{equation*} 
	and where $r_0$ and $f^\eps_i$ satisfy
	\begin{itemize}
		\item there exist $ \pare{\alpha_i}_{i\in\set{1, \ldots , q}}, \ \alpha_i \geqslant 0 $ such that
			\begin{align*}
				r_0\pare{n , \overrightarrow{k_q}} \leqslant C \prod_{i=1}^q \pare{1+ \av{k_i}}^{\alpha_i}, 
			\end{align*}
		\item $ \pare{\cF^{-1} f_i^\varepsilon}_{0<\varepsilon<\varepsilon_0} $ is uniformly bounded in $ \cC \pare{[0, T]; H^{\sigma + \alpha_i}\pare{\T^d}} $, for any $i \in \set{1, \ldots, q}$,
		\item there exists a $ \sigma_i >-\sigma $ for which, $ \pare{\cF^{-1} \partial_t f_i^\varepsilon}_{0<\varepsilon<\varepsilon_0} $ is uniformly bounded in $ \cC \pare{[0, T]; H^{\sigma_i}\pare{\T^d}} $.
	\end{itemize}
\end{definition}

\noindent The abstract concept in Definition \ref{def:oscillating_functions} is required in order to introduce the following result, see \cite[Lemma 2.1]{Gallagher_singular_hyperbolic} or \cite[Lemma 2.1]{Gallagher_incompressible_limit} for more details.
\begin{lemma}\label{lem:schochet_abstract}
	Let $ T>0 $ and $ \sigma > \dfrac{d}{2} +2 $, let $ \pare{b^\varepsilon}_{\varepsilon} $ be a family of functions, bounded in $ \cC\pare{[0, T]; H^\sigma \pare{\T^d}} $ and let $ a_0^\varepsilon\to 0 $ as $ \varepsilon \to 0 $ in $ H^{\sigma-1}\pare{\T^d} $. Let $ \cQ^\varepsilon, \ \cA^\varepsilon_2 $ be as in \eqref{eq:def_Qeps}, \eqref{eq:def_A2eps}, let $ R^\varepsilon_{\osc} $ be a $ \pare{p, \sigma-1} $--oscillating function and finaly let $ F^\varepsilon\to 0 $ as $ \varepsilon \to 0 $ in $ \cC \pare{[0, T]; H^{\sigma-1}\pare{\T^d}} $. Then the function $ a^\varepsilon $, solution of 
	\begin{equation*}
		\left\lbrace
		\begin{aligned}
			& \partial_t a^\varepsilon +\cQ^\varepsilon \pare{a^\varepsilon, b^\varepsilon}- \cA^\varepsilon_2 \pare{D} a^\varepsilon= R^\varepsilon_{\osc} + F^\varepsilon, \\
 			&\left. a^\varepsilon\right\vert _{t=0} = a^\varepsilon_0,
		\end{aligned}
		\right.
	\end{equation*}
	is an $ o_\varepsilon\pare{1} $ in the $ \cC \pare{[0, T]; H^{\sigma-1}\pare{\T^d}} $ topology.
\end{lemma}

\bigskip

Now, to prove our main result, we subtract \eqref{eq:limit_system} from \eqref{eq:filt-sys}, and we denote the difference unknown by $ W^\varepsilon=U^\varepsilon- {U} $. Some basic algebra calculations lead to the following difference system
\begin{equation} \label{equation_W_schochet_method}
 	\left\lbrace
 	\begin{aligned}
 		&\partial_t W^\varepsilon+ \mathcal{Q}^\varepsilon\left( W^\varepsilon, W^\varepsilon+2 {U} \right) - \cA_2^\varepsilon\pare{D} W^\varepsilon = -\pare{\cR^\varepsilon_{\osc} + \cS^\varepsilon_{\osc}}, \\
 		& \dive W^\varepsilon=0,\\
 		&\left. W^\varepsilon \right|_{t=0}= 0,
 	\end{aligned}
 	\right.
\end{equation}
where
\begin{align*}
	&\cR^\varepsilon_{\osc} = \mathcal{Q}^\varepsilon\left( {U} , {U}\right) - \mathcal{Q} \left(   {U},  {U} \right), \\
 	&\cS^\varepsilon_{\osc} = - \pare{\cA^\varepsilon_2\pare{D}-\cA^0_2\pare{D}}U.
\end{align*}

\noindent We remark that $ \cR^\varepsilon_{\osc} $ and $ \cS^\varepsilon_{\osc} $ are highly oscillating functions which converge to zero in $ \cD'\pare{\T^3\times \R_+ } $ only. Thanks to the results proved in Section \ref{se:lim_syst}, namely Lemma \ref{le:lim_smooth} and equation \eqref{eq:cQ_cA}, we can compute the explicit value of $ \cR^\varepsilon_{\osc} $ and $ \cS^\varepsilon_{\osc} $ which is given by
\begin{equation*}
	\cR^\varepsilon_{\osc} = \cR^\varepsilon_{\osc, \RN{1}} + \cR^\varepsilon_{\osc, \RN{2}} + \cR^\varepsilon_{\osc, \RN{3}}, 
\end{equation*}
and
\begin{align*}
	\cF \cR^\varepsilon_{\osc, \RN{1}} & =  \sum_{\substack{\omega^{a,b,c}_{k,m,n}\neq 0\\k+m=n \\ k_h, m_h, n_h\neq 0 \\ a,b,c \in \set{ 0,\pm}}} e^{i\frac{t}{\varepsilon} \omega^{a,b,c}_{k,m,n}} \left( \left.  \ \mathbb{P}_n \left(  n, 0 \right) \cdot \bS \left( U^{a} (k)  \otimes  U^b (m)\right) \right| e_c(n) \right)_{\mathbb{C}^4}\; e_c(n), \\
	\cF \cR^\varepsilon_{\osc, \RN{2}} & =  2\sum_{\substack{\pare{0,k_3} + m=n\\ m_h, n_h \neq 0\\ \widetilde{\omega}^{b,c}_{m,n} \neq 0\\b, c =0, \pm\\j=1,2,3}} e^{i\frac{t}{\varepsilon}\widetilde{\omega}^{b,c}_{m,n}} \left( \left.  \ \bP_n \left(  n, 0 \right) \cdot \bS \left( U^{j} (0,k_3)  \otimes  U^b (m)\right) \right| e_c(n) \right)_{\mathbb{C}^4}\; e_c(n), \\
	\cF \cR^\varepsilon_{\osc, \RN{3}} & =    \sum_{\substack{k+m=(0,n_3) \\ k_h, m_h \neq 0\\ \omega^{a, b}_{k, m}\neq 0 \\ a,b \in \set{ 0,\pm}\\ j=1,2,3}} e^{i\frac{t}{\varepsilon}\omega^{a, b}_{k, m}} \left( \left.  \ \mathbb{P}_{(0,n_3)} \left(0,0,  n_3, 0 \right) \cdot \bS \left( \widetilde{V}_1^{a} (k)  \otimes  \widetilde{V}_2^b (m)\right) \right| f_j \right)_{\mathbb{C}^4}\; f_j,\\
	\cF \cS^\varepsilon_{\osc} & = 1_{n_h\neq 0} \sum_{ \substack{\omega^{a,b}_n\neq 0\\a, b=0,\pm}} e^{i\frac{t}{\varepsilon}\omega^{a,b}_n} \left(\left. \cF \cA_2 (n) U^a (n) \right| e_b(n)  \right)_{\mathbb{C}^4} \ e_b(n). 
\end{align*}
The following result is immediate. 
\begin{prop}
	Under the assumption of Theorem \ref{thm:main_result} the function $ \cR^\varepsilon_{\osc} $ is a $ \left( 2, s-1 \right) $--oscillating function, $\cS^\varepsilon_{\osc} $ is a $ \left( 1, s-2 \right) $--oscillating function and hence $ \cR^\varepsilon_{\osc} + \cS^\varepsilon_{\osc} $ is a $ \left( 2, s-2 \right) $--oscillating function. 
\end{prop}

We can now conclude by applying Lemma \ref{lem:schochet_abstract}, with $ \sigma = s-2 $ and with the substitutions
\begin{align*}
	& a^\varepsilon = W^\varepsilon, & b^\varepsilon = W^\varepsilon +2U, \\
	& R^\varepsilon_{\osc} = -\pare{\cR^\varepsilon_{\osc} + \cS^\varepsilon_{\osc}}, & F^\varepsilon =0.
\end{align*}
We deduce that for each $ T\in \left[ 0, T^\star\right) $, the function $ W^\varepsilon $ is an $ o_{\varepsilon} \pare{1} $ function in $ \cC \pare{[0, T]; H^{s-2}} $. Setting hence
\begin{equation*}
	\widetilde{T^\star} = \sup \set{t\in [0, T^\star ) \ \Big\vert \ \norm{U^\varepsilon \pare{t'}}_{H^{s-2}}< K \pare{\Big. \cE_1 \pare{V_0} + \cE_{3, \nu, T}\pare{V_0}}, \ \forall \ t'\in \left[0, t\right]},
\end{equation*}
where $ \cE_1 $ and $ \cE_{3, \nu, T} $ are defined in Proposition \ref{pr:global_Hs_ubarh} and \ref{pr:global_Hs_uosc}, and $ K $ is a positive (possibly large) fixed, finite constant. Since $ W^\varepsilon = U^\varepsilon - U $ we deduce that for any $ t\in\bra{0, \widetilde{T^\star}} $
\begin{align*}
	\norm{U^\varepsilon \pare{t}}_{H^{s-2}}  & \leqslant \norm{U \pare{t}}_{H^{s-2}} + \norm{ W^\varepsilon \pare{t}}_{H^{s-2}}, \\
	& \leqslant \frac{K}{2}  \pare{\Big. \cE_1 \pare{V_0} + \cE_{3, \nu, T}\pare{V_0}} + \frac{1}{2},
\end{align*} 
since 
\begin{align*}
\norm{U \pare{t}}_{H^{s-2}} & \leqslant \norm{\bU\pare{t}}_{H^{s-2}} + \norm{U _{\osc} \pare{t}}_{H^{s-2}}  + \norm{\uU\pare{t}}_{H^{s-2}} , \\
& \leqslant \mathcal{E}_1\pare{V_0} + \mathcal{E}_{3, \nu, T}\pare{V_0} + \norm{\uU_0}_{H^{s-2}}, \\
& \leqslant \frac{K}{2}  \pare{\Big. \cE_1 \pare{V_0} + \cE_{3, \nu, T}\pare{V_0}}, 
\end{align*}
for $ K > 4 $ and $  \norm{W^\varepsilon \pare{t}}_{H^{s-2}}\leqslant 1/2 $ thanks to the result of Lemma \ref{lem:schochet_abstract}. 
Thus, $ \widetilde{T^\star}=T^\star $, and supposing $ T^\star < \infty $, we deduce
\begin{align*}
	\lim _{t\nearrow T^\star} \int_0^t \norm{\nabla U^\varepsilon\pare{t'}}_{L^\infty} dt' & \leqslant \lim _{t\nearrow T^\star} \int_0^t \norm{ U^\varepsilon\pare{t'}}_{H^{s-2}} dt' \\
	& \leqslant K \pare{\Big. \cE_1 \pare{V_0} + \cE_{3, \nu, T}\pare{V_0}} T^\star <\infty,
\end{align*}
which indeed contradicts \eqref{eq:BU_criterion}. We conclude that $ T^\star =\infty $.

\section*{Acknowledgments}

{The research of S.S. is supported by the Basque Government through the BERC 2018-2021 program and by Spanish Ministry of Economy and Competitiveness MINECO through BCAM Severo Ochoa excellence accreditation SEV-2013-0323 and through project MTM2017-82184-R funded by (AEI/FEDER, UE) and acronym "DESFLU".}

\footnotesize{
\providecommand{\bysame}{\leavevmode\hbox to3em{\hrulefill}\thinspace}
\providecommand{\MR}{\relax\ifhmode\unskip\space\fi MR }
\providecommand{\MRhref}[2]{%
  \href{http://www.ams.org/mathscinet-getitem?mr=#1}{#2}
}
\providecommand{\href}[2]{#2}
}
\end{document}